\documentclass[preprint,12pt,authoryear]{elsarticle}
\journal{}
\let\today\relax
\makeatletter
\def\ps@pprintTitle{%
    \let\@oddhead\@empty
    \let\@evenhead\@empty
    \def\@oddfoot{\footnotesize\itshape
         {} \hfill\today}%
    \let\@evenfoot\@oddfoot
    }
\makeatother
\usepackage{natbib}
\setcitestyle{numbers,square}
\usepackage{amsmath,amssymb,amsthm}
\usepackage{hyperref}
\usepackage{geometry}
\usepackage{mathrsfs}
\usepackage{graphicx}
\usepackage[font={small,it}]{caption}
\newtheorem{theorem}{Theorem}[section]

\newtheorem{proposition}[theorem]{Proposition}
\newtheorem{remark}[theorem]{Remark}
\newtheorem{lemma}[theorem]{Lemma}
\newtheorem{fact}[theorem]{Fact}
\newtheorem{corollary}[theorem]{Corollary}

\def\d{\mathrm{d}}

\begin{document}
\begin{frontmatter}
\title{Speed of Random Walks in Dirichlet Environment on a Galton-Watson Tree}
\author[label 1]{Dongjian Qian}
\author[label 1]{Yang Xiao}
\affiliation[label 1]{organization={School of Mathematical Sciences},
            addressline={Fudan University}, 
            city={Shanghai},
            country={People's Republic of China}}
	\begin{abstract}
		This paper deals with a transient random walk in Dirichlet environment, or equivalently a linearly edge reinforced random walk, on a Galton-Watson tree. We compute the stationary distribution of the environment seen from the particle of an edge reinforced random walk. We obtain a formula for the speed and give a necessary and sufficient condition for the walk to have a positive speed under some moment conditions on the offspring distribution of the tree.
\end{abstract}  
 \end{frontmatter}
\section{Introduction}
An edge reinforced random walk (ERRW) is a non-Markov process that tends to favor previously visited edges, first introduced by Coppersmith and Diaconis \cite{diaconis1986errw}. Let $G=(V,E)$ be an oriented graph and $(\alpha_e)_{e\in E}$ a set of positive deterministic weights. For two adjacent points $x,y\in V$, we denote by $(x,y)$ the edge from $x$ to $y$. An oriented edge $e$ is thus written as $e:=(\underline e,\overline e)$, where $\underline{e}$ and $\overline{e}$ are the tail and head of $e$ respectively. We define the ERRW $(X_n)_{n\ge 0}$ on $G$ with transition probabilities 
 	\begin{equation}\label{eq_defnERRW}
		\mathbf P(X_{n+1}=y|X_1,\cdots,X_n)=\frac{\alpha_{(X_n,y)}+N^{X}_{(X_n,y)}(n)}{\sum_{\underline{e}=X_n}\alpha_e+N^{X}_e(n)}\textbf{1}_{\{(X_n,y)\in E \}},
	\end{equation}
where $N_e^{X}(n):=\#\{1\leq k\leq n:(X_{k-1},X_{k})=e\}$. In words, if the process is at a vertex $x$ at time $n$, it will choose for its next step some neighbour $y$ with probability proportional to $\alpha_{(x,y)}+N^X_{(x,y)}(n)$.

By means of Polya's urns, it is well-known that this process can be represented as a mixture of Markov chains called Random Walk in Dirichlet Environment (RWDE), a special case of random walk in random environment. Specifically, independently at each vertex $x$, pick a random vector with positive entries $(\eta_e)_{\underline e=x}=(\eta(\underline e,\overline{e}))_{\underline e=x}$  which satisfies $\sum_{\underline{e}=x}\eta_e=1$.  The joint law of $(\eta_e)_{\underline e=x}$ is taken to be the Dirichlet distribution with parameters $(\alpha_e)_{\underline e=x}$, i.e. it has density 
	\begin{equation*}
		\frac{\Gamma(\sum_{\underline e=x} \alpha_e)}{\prod_{\underline e=x}\Gamma(\alpha_e)}\prod_{\underline e=x} (y_e^{\alpha_{e}-1}\textbf{1}_{\{0<y_e<1\}})\textbf{1}_{\{\sum _{\underline e=x}y_e=1\}}.
	\end{equation*}
We call $(\eta_e)_{e\in E}$ a Dirichlet environment, and denote its distribution by ${\rm DE}(\cdot)$.  Given the environment $(\eta_e)_{e\in E}$, we define the RWDE $(X_n)_{n\geq0}$ as the Markov chain on $G$ with transition probabilities
	\begin{equation*}
		{{\rm P}}^{\eta}(X_{n+1}=y|X_n=x)=\eta(x,y),\qquad (x,y)\in E. 
	\end{equation*}
We write ${\rm P}^{\eta}_x(\cdot)$ for the quenched measure ${\rm P}^{\eta}(\cdot|X_0=x) $. The connection between ERRW and RWDE is as follows: the RWDE $(X_n)_{n\geq0}$ under the annealed measure $\mathbf{P}(\cdot):=\int {{\rm P}}^\eta(\cdot){\rm DE}(\d \eta)$ is an ERRW.

In the article,  we focus on the case where $G$ is a super-critical Galton-Watson tree $\mathbb T$ with some offspring distribution $(p_n)_{n\geq0}$, and hence $m:=\sum_{n\geq0}np_n>1$. We consider it as an oriented graph where each edge has two directions (from parent to child and child to parent). We write $\rho$ for the root, and $xi,1\leq i\leq\nu(x)$, resp. $x_*$, for the children, resp. the parent, of a vertex $x$. Often, we will artificially add a parent $\rho_*$ to the root $\rho$, and we will call the new tree $\mathbb T_*$. Since we are interested in the transient case, we will preferably work on the event $\mathcal S$ that $\mathbb T_*$ is infinite. 

Fix two positive numbers $\alpha_p,\alpha_c$. We study the ERRW $(X_n)_{n\geq 0}$ on the tree $\mathbb T_*$ named $(\alpha_p,\alpha_c)$-ERRW where the weights $(\alpha_e)_{e\in E}$ are given by $\alpha_{(x,x_*)}=\alpha_p$ and $\alpha_{(x_*,x)}=\alpha_c$, $x\in \mathbb T$, with $X_0=\rho$. It is a generalization of the model given in \cite{pemantle1988phase}. Similar to the $\lambda$-biased random walk (see, for example, \cite{lyons1996biased}), the $(\alpha_p,\alpha_c)$-ERRW introduces an asymmetry between moving up or down in the tree.   By the connection mentioned above, $(X_n)_{n\geq 0}$ can be identified with an $(\alpha_p,\alpha_c)$-RWDE on $\mathbb T_*$. We denote by  $\mathbb P(\cdot)$ the annealed distribution of $(X_n)_{n\geq 0}$ when we also average over the tree $\mathbb T_*$, and by $\mathbb E$ the associated expectation. To sum up, the walk $(X_n)_{n\geq 0}$ can be studied under three levels of randomness: under ${\rm P}^{\eta}$ (both the tree $\mathbb T_*$ and the environment $(\eta_e)_{e\in E}$ are fixed, the walk is a Markov chain), under $\mathbf P$ (the tree $\mathbb T_*$ is fixed and the walk is an $(\alpha_p,\alpha_c)$-ERRW on $\mathbb T_*$), and under $\mathbb P$ where we average over everything. We point out that what we defined is a directed ERRW on a tree, which corresponds, in the setting of \cite{pemantle1988phase}, to an undirected ERRW with weights $(2\alpha_p-1,2\alpha_c)$. 

\medskip
Let $(A_i,1\leq i\leq\nu)$ have the distribution under $\mathbb P$ of 
	\begin{equation}\label{eq_Aisratioomega}
  \left(\frac{\eta(\rho,\rho1)}{\eta(\rho,\rho_*)},\cdots,\frac{\eta(\rho,\rho\nu(\rho))}{\eta(\rho,\rho_*)}\right).
	\end{equation}
From Lyons and Pemantle \cite{lyons1992random}, we know that the walk $(X_n)_{n\ge 0}$ is transient $\mathbb P(\cdot|\mathcal{S})$-a.s. if and only if
	\begin{equation}\label{eq_transience}
		\inf_{t\in[0,1]}\mathbb E\left[\sum_{i=1}^\nu A_i^t\right]>1,
	\end{equation} 
which is specified in Proposition \ref{prop_when_transient_RWDE} in our case. Let $|x|$ be the generation of the vertex $x$ and set $|\rho_*|:=-1$. Under $\mathbb P(\cdot|\mathcal{S})$, the quantity
\begin{equation*}
    v:= \lim_{n\to\infty}|X_n|/n
\end{equation*}    
is called {\bf speed} of the random walk $(X_n)_{n\geq0}$. This limit exists $\mathbb P(\cdot|\mathcal{S})$-a.s. indeed and is deterministic by \cite{gross2004marche}.

Assuming that the offspring distribution has high enough moments, our first result gives a necessary and sufficient condition for the speed $v$ to be positive, so that the walk drifts towards infinity linearly.
As far as we know, the only available result was in the case $\alpha_p=1$, $\alpha_c=1/2$ on a $d$-regular tree, $d\geq 2$  \cite{aidekon2008transient} (see also \cite{collevecchio2006limit}).
Note that the ratios $(A_i)_{1\le i\le \nu}$ in Dirichlet environment are not bounded from above and below, so the general criteria for random walks in random environment on Galton--Watson trees put forward in \cite{aidekon2008transient} cannot be applied. Let $f(s):=\sum_{n\geq0}s^{n}p_n$ denote the generating function of the offspring distribution and $q$ the smallest root of $f(s)=s$ in $[0,1]$, i.e. $q$ is the extinction probability $1-\mathbb P(\mathcal{S})$ of the Galton--Watson tree. We introduce 
		\begin{align*}
			 d:=\min\{n\ge 1:p_n>0 \},\,r:=\sup\{k:\mathbb E[A^{-k}]f^\prime(q)<1\},
	\end{align*}
where $A$ is a generic random variable distributed as $A_1$ conditioned on $\nu\ge 1$.
\begin{theorem}\label{positive speed Thm}
Suppose that $m\in (1,\infty)$, that the $(\alpha_p,\alpha_c)$-ERRW $(X_n)_{n\ge 0}$ is transient, i.e.  \eqref{eq_transience} holds, and that 
\begin{equation}\label{moment condition}
    \mathbb E[\nu^{d\alpha_c+3+\alpha_p}]<\infty .
\end{equation}
The speed $v$ is positive if and only if:
		\begin{itemize}
			\item $2r-\alpha_c+\alpha_p-1>0$ when $p_0=0,p_1>0$;
			\item $(2d-1)\alpha_c+\alpha_p-1> 0$ when $p_0=0,p_1=0$;
			\item $r-\alpha_c+\alpha_p-1>0$ when $p_0>0$.
		\end{itemize}
	\end{theorem}
	
 The second case is reminiscent of a criterion for positive speed given in \cite{sabot2013random} in the case of the lattice $\mathbb Z^k,k\geq3$ in which the role of the parameter $\kappa$ there would be played by the quantity $2(\alpha_p+d\alpha_c)-(\alpha_p+\alpha_c)=(2d-1)\alpha_c+\alpha_p$. The third case is equivalent to $ (\alpha_p-1)/\alpha_c>f^\prime(q)$. Taking $(\alpha_p,\alpha_c)=(\lambda \alpha, \alpha)$ and (informally) making $\alpha$ go to $+\infty$, it boils down to $\lambda> f'(q)$, which agrees with the result of  Lyons, Pemantle, and Peres \cite{lyons1996biased} in the case of $\lambda$-biased random walk.

	 \begin{figure}[ht]
 
    \centering
    \includegraphics[width=0.8\textwidth]{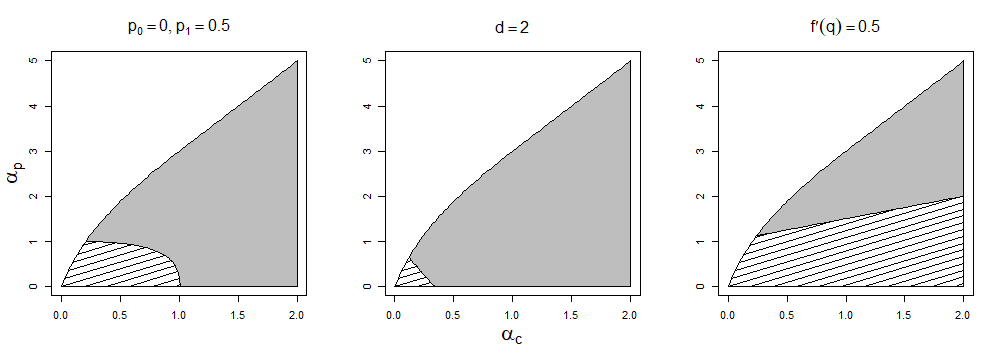}
    \caption{We have $m=2$ for these graphs. Gray areas are where $(\alpha_p,\alpha_c)$-ERRW has positive speed, while dashed areas are where the transient walk has zero speed. The remaining areas correspond to the region of recurrence. The first picture shows the effect of long pipe traps, which only depends on $p_1$. (Note that in this case $f^\prime (q)=p_1$.) The effect of traps is weak when $p_0=p_1=0$, as is shown in the second graph. Finally, we can see from the third one that traps from leaves have a stronger effect than traps from pipes.}
     \label{fig-threezone}
 \end{figure}

We actually obtain a necessary and sufficient condition (without condition \eqref{moment condition}) for the positivity of the speed in terms of the conductance of the tree.  Recall that we can view $(X_n)_n$ as a RWDE. Let $\tau_x:=\inf\{n\geq0:X_n=x\}$ denote the hitting time of $x$ with $\inf\emptyset=\infty$. For $x\in \mathbb T$, let
	\begin{equation}\label{eq_def_beta}
		\beta(x):={{\rm P}}_{x}^{\eta}(\tau_{x_*}=\infty),
	\end{equation}
 which is a functional of the tree $\mathbb T_*$ and the environment $(\eta_e)_e$. The random variable $\beta(\rho)$ is the so-called conductance of the tree $\mathbb T_*$. Since $(X_n)_{n\ge 0}$ is supposed to be transient, $\beta(\rho)>0$ a.s. conditioned on $\mathbb T_*$ being infinite.

 Define $\Phi:\mathbb N\rightarrow \mathbb R^+ $ as
\begin{equation}
\label{eq_phi_intro}
		\Phi(k):=\frac{\Gamma(\alpha_p)\Gamma(\alpha_c+k+1)}{\Gamma(\alpha_c+1)\Gamma(\alpha_p+k)}.
	\end{equation} 

 \noindent Our necessary and sufficient condition for the positivity of the speed reads as follows. 
\begin{theorem}\label{positive speed Thm 2}
    Suppose that $m\in (1,\infty)$ and that the $(\alpha_p,\alpha_c)$-ERRW  $(X_n)_{n\ge 0}$ is transient i.e.  \eqref{eq_transience} holds. The speed $v$ is positive if and only if 
    \begin{equation}\label{eq_assump_finite}
         \mathscr{C}:=\sum_{k\geq0} \Phi(k) \mathbb{E}\left[\frac{(1-\beta(\rho))^{k+1}}{\eta(\rho,\rho^*)}\right] \mathbb{E}\left[\beta(\rho)(1-\beta(\rho))^k \right]<\infty.
    \end{equation}
\end{theorem}

Our next result is a formula for the speed $v$ of $(X_n)_{n\ge 0}$.  In the standard case of $\lambda$-biased random walk on a Galton-Watson tree conditioned on non-extinction, Lyons, Pemantle, and Peres find an explicit speed formula when $\lambda=1$ \cite{lyons1995ergodic} with the help of an invariant measure. Note that the walk is a simple random walk in this case. For general bias $\lambda$, \cite{aidekon2014speed} gives a formula for the speed in terms of the conductance of a tree. We give a general formula for the speed of the $(\alpha_p,\alpha_c)$-ERRW $(X_n)_n$ which, as in \cite{aidekon2014speed}, involves the conductance $\beta(\rho)$. 
	
For sake of concision, we let $\beta$ be a generic random variable with distribution $\beta(\rho)$. Recall the definition of $(A_i)_{1\le i\le \nu}$ in \eqref{eq_Aisratioomega} and let $\beta_0,\beta_1,\cdots$ be  i.i.d. random variables distributed as $\beta$, and independent of $(A_i)_{1\le i\le \nu}$. We introduce the hypergeometric function (see \cite{bateman1953higher} for example) $_2F_1(1,\alpha_c+1;\alpha_p;x):[0,1)\rightarrow \mathbb R^+$:
	\begin{equation}\label{eq_def_F}
		_2F_1(1,\alpha_c+1;\alpha_p;x):=\sum_{k\geq0}\Phi(k)x^k,\,x\in[0,1).
	\end{equation} 
 We simply write $F(x)$ when there is no confusion.
 
 \begin{theorem}\label{speed formula introduction}  
Suppose that $m\in (1,\infty)$, that the $(\alpha_p,\alpha_c)$-ERRW  $(X_n)_{n\ge 0}$ is transient i.e. \eqref{eq_transience} holds, and that \eqref{eq_assump_finite} holds.  Then the speed $v$ is positive and can be expressed as
		\begin{equation}\label{eq_doubletree}
			\begin{split}
				&\mathbb{E}\left[\frac{\beta_0(\sum_{i=1}^{\nu}A_i+1)}{1+\sum_{i=1}^{\nu}A_i\beta_i}\times F\left(\frac{1-\beta_0}{1+\sum_{i=1}^{\nu}A_i\beta_i}\right)\right]^{-1}\mathbb{E}\left[\frac{\beta_0(\sum_{i=1}^{\nu}A_i-1)}{1+\sum_{i=1}^{\nu}A_i\beta_i}\times F\left(\frac{1-\beta_0}{1+\sum_{i=1}^{\nu}A_i\beta_i}\right)\right].
			\end{split}	
		\end{equation}
\end{theorem}
\begin{corollary}\label{cor_formula}
When $\alpha_p=\alpha_c=\alpha>0$, if \eqref{eq_transience} and \eqref{eq_assump_finite} hold, the speed formula reduces to 
\begin{equation}\label{eq_speed}
			\frac{\mathbb E\left[\beta_0{(\sum_{i=1}^\nu A_i-1)((1-\beta_0)\alpha^{-1}+\beta_0+\sum_{i=1}^\nu A_i\beta_i)}{(\beta_0+\sum_{i=1}^\nu A_i\beta_i)^{-2}}\right]}{\mathbb E\left[\beta_0(\sum_{i=1}^\nu A_i+1){((1-\beta_0)\alpha^{-1}+\beta_0+\sum_{i=1}^\nu A_i\beta_i)}{(\beta_0+\sum_{i=1}^\nu A_i\beta_i)^{-2}}\right]}.
		\end{equation}
 When $\alpha_p=1$, $\alpha_c=1/2$ and \eqref{eq_transience} holds, it reduces to
		\begin{equation}
			\label{eq_speed_errw}
			\frac{\mathbb E\left[\beta_0(\sum_{i=1}^\nu A_i-1)\sqrt{1+\sum_{i=1}^{\nu}A_i\beta_i}(\beta_0+\sum_{i=1}^{\nu}A_i\beta_i)^{-3/2}\right]}{\mathbb E\left[\beta_0(\sum_{i=1}^\nu A_i+1)\sqrt{1+\sum_{i=1}^{\nu}A_i\beta_i}(\beta_0+\sum_{i=1}^{\nu}A_i\beta_i)^{-3/2}\right]}.
		\end{equation}
\end{corollary}
	
The main tool of our paper is the invariant measure of the environment seen from a particle, which is standard in the theory of random walks in random environment. Using arguments from ergodic theory, this invariantinvariantinvariant measure describes the asymptotic distribution of the environment seen from the particle. The stationary measures have been investigated in several models, like random walk in random environment (RWRE) on $\mathbb Z$,  see \cite{bolthausen2002ten} for example; simple random walks on a Galton-Watson tree in \cite{lyons1995ergodic}; $\lambda$-biased random walks on a Galton-Watson tree in \cite{aidekon2014speed}, \cite{lin2019harmonic}; null recurrent biased random walks or RWRE on a Galton-Watson tree in \cite{peres2008central}, \cite{faraud2011central}; continuous-time biased random walks on a Galton-Watson tree in \cite{ben2013einstein}; RWDE on $\mathbb{Z}^k,k\geq3$ in \cite{sabot2013random}, etc.  

In this article, we give an expression of the invariant distribution of the environment seen from $(X)_{n\ge 0}$. To do so, we need the concept of double trees marked with a distinguished path. In words, double trees consist of the gluing of two trees, the tree $\mathbb T^+_*$ standing for the subtree rooted at the particle, and the tree $\mathbb T^-_*$ standing for the part of the tree located below the particle. The distinguished path $\gamma$ represents the history of the walk up to the current time.

Let us give the stationary measure for the environment seen from the particle. We first sample two independent Galton-Watson trees $\mathbb T^-_*,\mathbb T^+_*$, with the $(\alpha_p,\alpha_c)$-initial weights. Then we create the double tree $\mathbb T^-\rightleftharpoons\mathbb T^+$ by connecting the roots of  $\mathbb T^-_*,\mathbb T^+_*$, denoted by $\rho^-,\rho^+$ respectively (see Figure \ref{fig:double_tree}). In other words, we let $\rho^-_*:=\rho^+,\rho^+_*:=\rho^-$. We require that $\alpha_{(\rho^-,\rho^+)}=\alpha_{(\rho^+,\rho^-)}=\alpha_p$ and the other weights are the same as those on the original trees. We call $e_{\rho}:=(\rho^-,\rho^+)$ the root edge. 
 \begin{figure}[ht]
    \centering
    \includegraphics[width=0.75\linewidth]{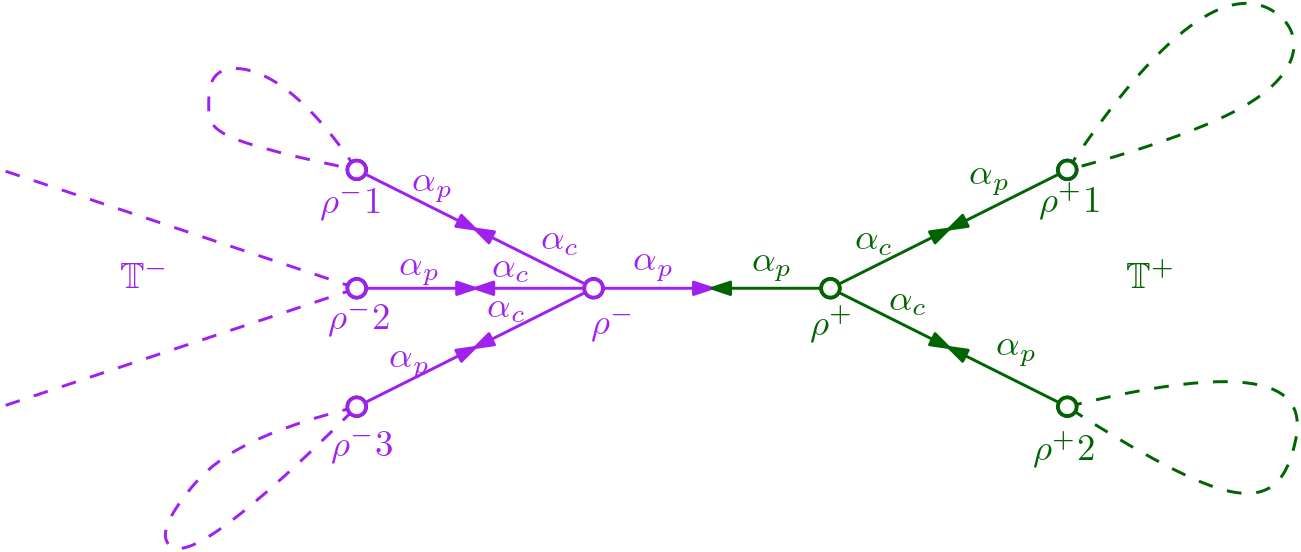}
    \caption{A double tree with $(\alpha_p,\alpha_c)$-initial weights. We use purple and green colors to distinguish $\mathbb T^-$ and $\mathbb T^+$. One can see that the double tree weights are locally the same as those on a Galton-Watson tree.}
    \label{fig:double_tree}
 \end{figure}

The double tree is a weighted directed graph, so we can define a Dirichlet environment on it (see Figure \ref{fig:double_tree}). Let $Y=(Y_n)_{n\ge 0}$ and $X=(X_n)_{n\ge 0}$ be two random walks on the double tree starting from $\rho^+$ and $\rho^-$ respectively which are conditionally independent given the Dirichlet environment (see Figure \ref{fig:XYonDoubleTree}). We stress that, after averaging over the Dirichlet environment, $Y$ and $X$ are not independent anymore. Let $Rev((X_n)_{n\ge 0})$ be the time-reverse of the path of $(X_n)_{n\ge 0}$, which is therefore indexed by $\mathbb{Z}_-$.

For $k\geq 0$, we concatenate the reversed path $Rev((X_n)_{n\ge 0})$ with the finite path $(Y_l)_{0\le n\le k}$ and denote it by $Rev((X_n)_{n\ge 0})\ast(Y_n)_{0\leq n\leq k}$.  Since $X$ is transient, this path  `comes' from the boundary of the tree $\mathbb{T}^-$ . Let $\tau^X_{x}:=\inf\{n> 0, X_n=x\}$ be the hitting time of a vertex $x$ on the double tree for $X$. We call $\mu_{ER}^{(k)}$ the distribution of the double tree with marked path 
\begin{align*}
   \left (\mathbb T^-\rightleftharpoons\mathbb T^+, Rev((X_n)_{n\ge 0})\ast(Y_n)_{0\leq n\leq k}\right).
\end{align*}
We define the distribution $\mu_{ER}$ on the space of double trees with marked path by
\begin{align}\label{eq_def_muER}
    \mu_{ER}(\cdot)=\sum_{k=1}^{\infty} \Phi(N^{Y}_{e_{\rho}}(k))\textbf{1}_{\{\tau^X_{\rho^+}=\infty,Y_k=\rho^+\}} \mu_{ER}^{(k)}(\cdot)
\end{align}
where the definition of the function $\Phi: \mathbb N \rightarrow \mathbb R$ is given in (\ref{eq_phi_intro}). It defines an invariant measure for the environment seen from the particle for the $(\alpha_p,\alpha_c)$-ERRW (See Corollary \ref{Cor_inv_meas}).

\begin{figure}[ht]
 
    \centering
    \includegraphics[width=0.8\textwidth]{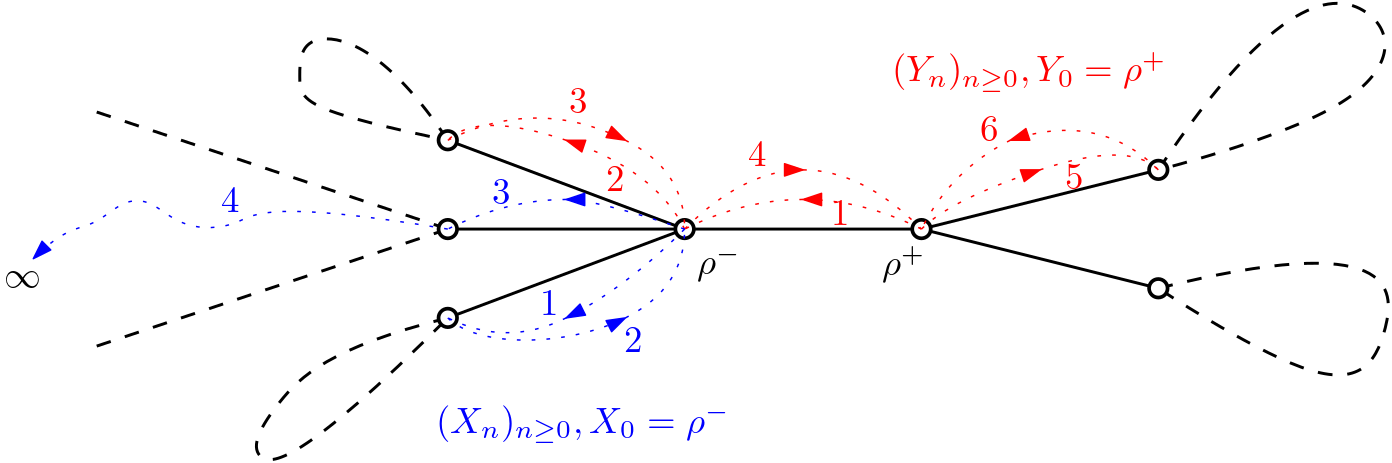}
    \caption{Two random walks $X$ and $Y$ on the double tree with weights in Figure \ref{fig:double_tree}. After sampling the double tree and Dirichlet environment, we sample $(X_n)_{n\geq0},(Y_n)_{n\geq0}$ independently. We use blue color to indicate that we need to reverse the path of $(X_n)_{n\geq0}$.}
    \label{fig:XYonDoubleTree}
 \end{figure}
 
The article is arranged as follows. In Section \ref{section_preliminaries}, we talk about regeneration structure and give some properties of RWDE, especially RWDE on a Galton-Watson tree. In Section \ref{section_asy}, we extend the path reversal argument of \cite{aidekon2014speed} to our case and characterize the environment seen from a particle when it is far away from the root. We also prove Theorem \ref{positive speed Thm 2} in this section and give the invariant measure. In Section \ref{section_speed}, we prove Theorem \ref{speed formula introduction} and give the method of computing (\ref{eq_speed}) and (\ref{eq_speed_errw}). We also deduce the criteria for positive speed and prove Theorem \ref{positive speed Thm}. 
\\
\textbf{Acknowledgements:}
	We would like to thank Elie Aïdékon for offering the question and many important discussions.
	\section{Preliminaries}\label{section_preliminaries}
	\subsection{Facts about Regeneration Time} \label{Sec_Fact_Reg}
	For a random walk $(X_n)_{n\ge 0}$ on a Galton-Watson tree starting from $\rho$, we call $\theta$ a fresh epoch if $X_\theta\not =X_n$ for all $n<\theta$ and a regeneration epoch if additionally, $X_{\theta-1}\not= X_n$ for all $n\ge\theta$. More specifically, let $\theta_0:=0,\Theta_0:=0$. For $k\geq1$, let
	\begin{equation*}
		\theta_k:=\inf\{n>\theta_{k-1}:X_n\not=X_j,\forall0\leq j<n\}
	\end{equation*}
	be the $k$-th fresh epoch, and
	\begin{equation}\label{def_reg_epoch}
		\Theta_k:=\inf\{n>\Theta_{k-1}:n\in\{\theta_i,i\geq1\},X_j\not=(X_n)_*,\forall j\geq n\}
	\end{equation}
	be the $k$-th regeneration epoch. In \cite[Section 3]{lyons1996biased}, properties of these random times for biased random walk are given and these facts can be enhanced to the proposition below.
	\begin{proposition}\label{prop_infiregiid}
		For any transient random walk in random environment on a Galton-Watson tree given non-extinction, there are infinitely many regeneration epochs a.s. Moreover $\{\Theta_{k+1}-\Theta_k\}_{k\geq1}$ are i.i.d. as are the increments $\{|X_{\Theta_{k+1}}|-|X_{\Theta_k}|\}_{k\geq1}$.
\end{proposition}
	We omit the proof of the proposition since it is similar to the one in \cite{lyons1996biased}. 
 
Let $\mathbb E_{\rho}[\cdot]=\mathbb E[\cdot|X_0=\rho]$. As proved in \cite{gross2004marche}, the speed of random walk in random environment is a.s. the deterministic constant
 \begin{align}\label{eq_speed_v}
     \dfrac{\mathbb E_{\rho}[|X_{\Theta_{2}}|-|X_{\Theta_1}||\mathcal{S}]}{\mathbb E_{\rho}[\Theta_{2}-\Theta_1|\mathcal{S}] },
 \end{align}
where the numerator $\mathbb E_{\rho}[|X_{\Theta_{2}}|-|X_{\Theta_1}||\mathcal{S}]$ is always finite. Therefore, the speed is positive if and only if $\mathbb E_{\rho}[\Theta_{2}-\Theta_1|\mathcal{S}]<\infty$. We sum up these results as a fact.

	\begin{fact}\label{lemma_regspeed}
		For a transient RWRE $(X_n)_{n\ge 0}$ on a Galton-Watson tree, the following statements are equivalent:
		\begin{enumerate}
			\item $\lim_{n\to\infty} |X_n|/n$ is deterministic and positive;
			\item $\mathbb E_{\rho}[\Theta_2-\Theta_1|\mathcal{S}]<\infty$.

		\end{enumerate}
	\end{fact}
   In particular, the speed is deterministic for RWDE, and we deduce that the speed of ERRW is identical to that of RWDE. 
 
	\subsection{Facts about RWDE}\label{S_Facts_RWDE}
	This section covers some basic properties of Dirichlet environment. See \cite{sabot2017random} for a more thorough review.
	
 Let $p$ be a fixed positive integer and $(\alpha_1,\cdots,\alpha_p)$ be positive real numbers. We say that the random vector $D(\alpha_1\cdots \alpha_p):=(D_1,\cdots,D_p)$ has Dirichlet distribution with parameter $(\alpha_1,\cdots,\alpha_p)$ if it has density
	\begin{equation}\label{eq_density_Dirichlet}
		\frac{\Gamma(\sum_{j=1}^p \alpha_j)}{\prod_{j=1}^p\Gamma(\alpha_j)}\prod_{j=1}^p (x_j^{\alpha_j-1}\textbf{1}_{\{0<x_j<1\}})\textbf{1}_{\{\sum x_j=1\}}.
	\end{equation}
We denote by $Gamma(\alpha,r)$ the gamma distribution with parameters $\alpha$ and  $r$ whose density is 
\begin{align*}
    \frac{r^{\alpha}}{\Gamma(\alpha)}x^{\alpha-1}e^{-rx}\textbf{1}_{\{x\ge 0\}}.
\end{align*}
 By computation, we have the following lemma.
\begin{lemma}\label{prop_RWDEprop}	
If  $\Gamma_i\sim Gamma(\alpha_i,1),1\leq i\leq p$ are mutually independent random variables, then $(\Gamma_i/\sum_{j=1}^{p}\Gamma_j)_{1\le i\leq p}\sim D(\alpha_1\cdots \alpha_p)$. Besides, for $\alpha_j+s_j>0, 1\le j\le p$,
		\begin{equation*}
			{{\mathbf E}}\left[\prod_{j=1}^p (D_j)^{s_j}\right]=\frac{\Gamma(\sum_{j=1}^p \alpha_j)}{\prod_{j=1}^p\Gamma(\alpha_j)}\frac{\prod_{j=1}^p\Gamma(\alpha_j+s_j)}{\Gamma(\sum_{j=1}^p \alpha_j+\sum_{j=1}^p s_j)}.
		\end{equation*}
	\end{lemma}

As in the beginning of the introduction, let $G=(V,E)$ be an oriented graph and $(\alpha_e)_{e\in E}$ a set of positive deterministic weights. Recall that (\ref{eq_defnERRW}) defines an ERRW on $G$, a self-interacting model without the Markov property, under the measure $\mathbf{P}$. For a path $\gamma$ and an edge $e$, we denote by $N_e(\gamma)$ the local time of $e$ in $\gamma$ which is the number of times that $e$ appears in $\gamma$. 
 \begin{lemma}\label{lemma_ERRWpath}
    For a path $\gamma$ with length $n$ (i.e. with $n$ edges) on $G$,
		\begin{equation}
			\label{eq_p_path_a}
			\mathbf P((X_k)_{0\leq k\leq n}=\gamma)=\prod_{y\in V}\frac{\Gamma(\sum_{\underline{e}=y}\alpha_e)}{\Gamma(\sum_{\underline{e}=y}\alpha_e+N_e(\gamma))}\prod_{ e\in E}\frac{\Gamma(\alpha_e+N_e(\gamma))}{\Gamma(\alpha_e)}.
		\end{equation}
	\end{lemma}
The right-hand side of \eqref{eq_p_path_a} is well-defined by multiplying terms of all vertices and edges on the whole graph instead of just those in $\gamma$, since for $e\notin \gamma$, $N_e(\gamma)=0$.

Next, we consider two paths on $G$. Let $\alpha+N(\gamma)$ denote the weights $(\alpha_e+N_e(\gamma))_{e\in E}$. When the graph $ G$ and weights $(\alpha_e)_{e\in E}$ are fixed, we define ${\rm DE}^{G}(\cdot|\alpha)$ (resp. $\mathbf{E}_{DE}^{G}(\cdot|\alpha)$) as the probability measure (resp. expectation) corresponding to the Dirichlet environment on $G$ with weights $(\alpha_e)_{e\in E}$. Sometimes we write $\mathbf{E}_{DE}(\cdot|\alpha)$ (resp. $\mathbf{E}_{DE}^{ G}(\cdot)$) when the graph (resp. weights) is clear from the context.
	\begin{lemma}\label{lemma_path_sym}
 Let $G=(V,E)$ be an arbitrary graph. For two paths $\gamma_1,\gamma_2$ on $ G$ with weight $(\alpha_e)_{e\in  E}$, we have
		\begin{equation*}
			\mathbf E_{DE}[{{\rm P}}^\eta(\gamma_1){{\rm P}}^\eta(\gamma_2)|\alpha]=\mathbf P(\gamma_2|\alpha+N(\gamma_1))\mathbf P(\gamma_1|\alpha)=\mathbf P(\gamma_1|\alpha+N(\gamma_2))\mathbf P(\gamma_2|\alpha).
		\end{equation*}
	\end{lemma}
	\begin{proof}
		We only need to prove the first equality.
		\begin{equation*}
			\begin{split}
				&\mathbf E_{DE}[{{\rm P}}^\eta(\gamma_1){{\rm P}}^\eta(\gamma_2)|\alpha]=\mathbf E_{DE}\left[\prod_{ e\in E }\eta_e^{N_e(\gamma_1)+N_e(\gamma_2)}\right]\\
				&=\prod_{y\in V}\frac{\Gamma(\sum_{\underline{e}=y}\alpha_e)}{\Gamma(\sum_{\underline{e}=y}\alpha_e+N_e(\gamma_1)+N_e(\gamma_2))}\prod_{ e\in E}\frac{\Gamma(\alpha_e+N_e(\gamma_1)+N_e(\gamma_2))}{\Gamma(\alpha_e)}\\
				&={\mathbf P}(\gamma_1|\alpha)\prod_{y\in V}\frac{\Gamma(\sum_{\underline{e}=y}\alpha_e+N_e(\gamma_1))}{\Gamma(\sum_{\underline{e}=y}\alpha_e+N_e(\gamma_1)+N_e(\gamma_2))}\prod_{ e\in E}\frac{\Gamma(\alpha_e+N_e(\gamma_1)+N_e(\gamma_2))}{\Gamma(\alpha_e+N_e(\gamma_1))}\\
				&={\mathbf P}(\gamma_2|\alpha+N(\gamma_1)){\mathbf P}(\gamma_1|\alpha).
			\end{split}
		\end{equation*}
\end{proof}
	Lemma \ref{lemma_path_sym} shows that, in the same environment, the information one walk gives to the other is visualized as edge local times added to the weights after averaging the environment. 
 
 For completeness, we give a specific account of conditions on transience of RWDE on a Galton-Watson tree. Recall that $A_i,1\leq i\leq\nu$ are, according to (\ref{eq_Aisratioomega}), ratios of Dirichlet random variables.
 Lyons and Pemantle \cite{lyons1992random} show that a random walk in random environment is transient if and only if $\inf_{t\in[0,1]}\mathbb E[\sum_{i=1}^{\nu}A_i^t]>1$, that is in our case 
	\begin{equation}\label{eq_condition_recurrent}
		\inf_{t\in[0,1\wedge\alpha_p]}\frac{\Gamma(\alpha_p-t)\Gamma(\alpha_c+t)}{\Gamma(\alpha_c)\Gamma(\alpha_p)}>\frac{1}{m}.
	\end{equation}
	By taking the logarithm derivative with respect to $t$ on the left-hand side, we get $-\psi(\alpha_p-t)+\psi(\alpha_c+t)$, where $\psi:=\Gamma^\prime/\Gamma$ is a strictly increasing function named digamma function. Hence the minimum is attained at $t=0\vee\frac{\alpha_p-\alpha_c}2\wedge1$. We consider three cases: $\alpha_p\le \alpha_c$, $\alpha_p\ge\alpha_c+2$ and $\alpha_p\in (\alpha_c,\alpha_c+2)$ separately.
 \begin{itemize}
     \item When $\alpha_p\le \alpha_c$, the minimum is reached at $0$ and RWDE is always transient.  
     \item When $\alpha_p\ge\alpha_c+2$, the minimum is reached at $1$, so RWDE is transient if and only if $\alpha_p<m\alpha_c+1$. In this case, $\alpha_c\leq\frac1{m-1}$ always implies recurrence.
     \item Finally let us focus on the last case $\alpha_p\in(\alpha_c,\alpha_c+2)$, where the minimum is reached at $\frac{\alpha_p-\alpha_c}2$. We plug $t= \frac{\alpha_p-\alpha_c}2$ into (\ref{eq_condition_recurrent}) with $\alpha_c$ fixed, obtaining a function $g(\alpha_p)$. By taking the logarithm derivative of $g(\alpha_p)$ with respect to $\alpha_p$, we have
	\begin{equation*}
		\frac{\d}{\d \alpha_p}\log g(\alpha_p)=\frac{\d}{\d \alpha_p}\log\left(\frac{\Gamma(\frac{\alpha_p+\alpha_c}2)^2}{\Gamma(\alpha_p)\Gamma(\alpha_c)}\right)=\psi\left(\frac{\alpha_p+\alpha_c}2\right)-\psi(\alpha_p)\leq0.
	\end{equation*}
	At once we see that the minimum of $g(\alpha_p)$ for all $\alpha_p\in[\alpha_c,\alpha_c+2]$ is reached at $\alpha_p=\alpha_c+2$. We take $\alpha_p=\alpha_c+2$ and then $g(\alpha_p)=g(\alpha_c+2)>\frac{1}{m}$ becomes  $\alpha_c>\frac{1}{m-1}$. Therefore, when $\frac1{m-1}<\alpha_c<\alpha_p<\alpha_c+2$, the walk is always transient. When $\alpha_c\leq \frac1{m-1}$ there will be a critical point $\phi_0(\alpha_c)$ such that $g(\phi_0(\alpha_c))=\frac1m$. The function $g(\alpha_p)$ is decreasing with respect to $\alpha_p$, so the walk is transient only when $\alpha_p<\phi_0(\alpha_c)$.
 \end{itemize}

 In summary, we have the following proposition.
 \begin{proposition}
 \label{prop_when_transient_RWDE}
		Let $(X_n)_{n\ge 0}$ be a random walk in an $(\alpha_p,\alpha_c)$-Dirichlet environment on a Galton-Watson tree conditioned on non-extinction.
		\begin{itemize}
			\item[(a)] When $\alpha_c>\frac1{m-1}$, if $\alpha_p<m\alpha_c+1$ then RWDE is transient and otherwise recurrent;
			\item	[(b)] when $\alpha_c\le \frac1{m-1}$, if $\alpha_p<\phi_0(\alpha_c)$ then RWDE is transient and otherwise recurrent, where $\phi_0(\alpha_c)\in(\alpha_c,\alpha_c+2]$ satisfies
			\begin{equation*}
				\frac{\Gamma(\frac{\phi_0(\alpha_c)+\alpha_c}2)^2}{\Gamma(\phi_0(\alpha_c))\Gamma(\alpha_c)}=\frac1m.
			\end{equation*}
		\end{itemize}
	\end{proposition}
	
	\section{Asymptotic distribution of the environment seen
from the particle}\label{section_asy}
 
 	\subsection{Trees, Paths and Weights} \label{sec_treesWeights}
Following Neveu \cite{neveu1986arbres}, let $\mathcal{U}:=\{\rho\}\cup\bigcup_{n\geq1}(\mathbb{N}^+)^n$ denote the set of words and $\rho$ serves as an empty word. If $u\neq \rho$, we denote by $u_*$ the parent of $u$, which is the word $\rho i_1 i_2 \dots i_{n-1}$ if $u=\rho i_1 i_2\dots i_n$. We define $T$ as a subset of $\mathcal{U}$ s.t.
	\begin{itemize}
		\item $\rho\in T$,
		\item if $x\in  T\backslash\{\rho\}$, then $x_*\in T$,
		\item if $x = \rho i_1\cdots i_n\in  T\backslash\{\rho\}$, then any word $\rho i_1\cdots i_{n-1}j$ with $j\leq i_n$ belongs to $T$.
	\end{itemize}
Given $x,y\in\mathcal U$, we write $x\preceq y$ if  $x=y$ or $y=xj_1j_2\cdots j_n$ i.e. $x$ is an ancestor of $y$. Set $x\prec y$ if $x\preceq y$ and $x\not=y$. We create a new tree called $T_*$ by adding a vertex $\rho_*$ to $T$ as the parent of $\rho$. For all $x\in T$,  $\rho_*\prec\rho\preceq x$.
 
Then we define a double tree $ T^-\rightleftharpoons  T^+$ introduced in \cite[Section 2.2]{aidekon2014speed}. Intuitively, we first pick two trees $T^-$ and $T^+$ with root $\rho^-$ and $\rho^+$, and then artificially connect $\rho^-$ and $\rho^+$. See Figure \ref{fig:double_tree}. We use the double tree to study the environment seen from the particle. Suppose a transient random walk starts from the root $\rho$ and moves on the tree $T$. After a fairly long time, the walk would be far away from the root, and the tree seen from the current location looks like a double tree $T^-\rightleftharpoons T^+$. The starting point of the walk is somewhere high on $T^-$. We call the random double tree a Galton-Watson double tree if $\mathbb T^-$ and $\mathbb T^+$ are i.i.d Galton-Watson trees.

Now we represent vertices on a double tree by two sets of words $\mathcal U^+$ and $\mathcal U^-$. Specifically, we denote by $\rho^+i_1\cdots i_n\in\mathcal U^+$ (resp. $\rho^-i_1\cdots i_n\in\mathcal U^-$ ) the vertex on the double tree $T^-\rightleftharpoons T^+\subset \mathcal U^-\cup\mathcal U^+$ corresponding to $\rho i_1\cdots i_n$ on $T^+$ (resp. $T^-$ ). If $x\in  T^+\backslash \{\rho^+\}$ (resp. $ T^-\backslash \{\rho^-\}$), we set the parent of $x$, also denoted by $x_*$, as its parent on $ T^+$ (resp. $ T^-$). We also assume that $(\rho^+)_*=\rho^-$ and $(\rho^-)_*=\rho^+$. At last, we stress that afterward we always refer to a part of the corresponding double tree when we say $x\in T^+$, i.e. $x$ is in the $T^+$ part of $T^-\rightleftharpoons T^+$. 

Both RWDE and ERRW require specific parameter settings called edge weights. Parameters $(\alpha_e)_e$ on a double tree $ T^-\rightleftharpoons T^+$ should satisfy
	\begin{equation}\label{eq_weightofDT}
		\alpha_e=\left\{\begin{array}{ll}
			\alpha_p & e=(x,x_*) \text{ for some }x\in T^-\rightleftharpoons T^+,\\
			\alpha_c & \text{otherwise}.
		\end{array}\right.
	\end{equation}
We say a double tree has $(\alpha_p,\alpha_c)$ environment if the weights of edges satisfy (\ref{eq_weightofDT}). 

For  each vertex $x\in T^-\rightleftharpoons T^+$, we assign a group of Dirichlet random variables $(\eta_e)_{\underline e=x}$ with parameters $(\alpha_e)_{\underline e=x}$ to edges starting from $x$. Note that for an arbitrary vertex $x\in  T^-\rightleftharpoons  T^+$, there is only one edge starting from $x$ having weight $\alpha_p$. 

The transition law of ERRW is determined by the trajectory of the walk. We call a sequence of words $\gamma=(y_i)_{a<i<b},y_i\in \mathcal U^-\cup\mathcal U^+$ a path if $y_i$ and $y_{i-1}$ are adjacent for $a<i-1<i<b$, where $a,b \in \mathbb R\cup \{-\infty,\infty\}$. We say $\gamma$ is left-finite (resp. right-finite) if $a>-\infty$ (resp. $b<\infty$). Here we do not care about specific indices of the path, i.e. $(y_i)_{i\leq0}$ and $(y_{i-1})_{i\leq1}$ are considered the same. We define 
$$\Omega_{\gamma}:=\{\gamma=(y_i)_{i\le 0},y_0=\rho^+;\exists M, s.t.\text{ for } i<M ,y_i=\rho^-y_i^{(1)}\dots y_i^{(n_i)}\text{ and }n_i\rightarrow\infty \text{ as }i\rightarrow -\infty \}.$$
The set $\Omega_\gamma$ contains paths coming from one infinite end of $T^-$.

For a right-finite $\gamma^{(a)}$ and a left-finite $\gamma^{(b)}$, such that the last word of $\gamma^{(a)}$ is adjacent to the first word of $\gamma^{(b)}$, we can define their concatenation, denoted by $\gamma^{(a)}\ast\gamma^{(b)}$. More precisely, if $\gamma^{(a)}=(y^{(a)}_i)_{i\leq0}$ and $\gamma^{(b)}=(y^{(b)}_i)_{i\ge 0}$, let $\gamma^{(a)}\ast\gamma^{(b)}:=(\dots, y^{(a)}_{-1},y^{(a)}_0,y^{(b)}_0,y^{(b)}_1,\dots )$. We also define $Rev(\gamma)$ as the reverse of a path $\gamma$: if $\gamma=(y_i)_{i\geq0}$, then $Rev(\gamma):=(y_{-i})_{i\leq0}$. 
	
\subsection{Environment Seen from Fresh Point}\label{S_Rev}
Before showing the convergence in distribution of what a particle sees, we first present some intuitive lemmas similar to those in \cite{{aidekon2014speed}}. Recall that we define $T_*$ as $T\cup \{\rho_*\}$.  Given a word $x\in\mathcal U$, let $T_x$ be the subtree in $T_*$ rooted at $x$ (that consists of words $y\in \mathcal{U}$ such that $xy\in T$) and $xT_x$ be the tree composed of words $\{xy,y\in T_x\}$.  Also, set $T^{< x}_*$ as the tree obtained by removing $xT_x$ from $T_*$. We denote by $T_*^{\le x}$  the tree obtained from $T^{<x}_*$ by adding the word $x$. If $x\notin T$, let $T_x$ and $ T^{\leq x}_*$ be the empty set. 
	
Now we define $\Psi_x$ as a map from a tree $T_*$ to a double tree that preserves the connection of vertices. We map $x_*$ to $\rho^-$and $x$ to $\rho^+$, that is to say, the root edge of $\Psi_x(T_*) $ is $(\Psi_x(x_*),\Psi_x(x))$. The image of $T_*^{\leq x}$, $\Psi_x(T_*^{\leq x})$, can be seen as the backward tree at $x$ defined in \cite[Section 2.3]{aidekon2014speed}. As for $xT_x$, a vertex $y=xi_1\dots i_n$ is mapped to $\Psi_x(xy):=\rho^+i_1\dots i_n$. An intuitive picture of $\Psi_x$ is to hang the tree at vertex $x$ instead of $\rho$. We denote by $\hat{x}$ the word $\Psi_x(\rho_*)$. When $x\not\in T_*$, let $\Psi_x(T_*)$ be the empty set.
	\begin{lemma}\cite[Lemma 2.1]{aidekon2014speed}\label{lemma_backwardpsi}
		Given $x\in \mathcal{U}$ and a Galton-Watson tree $\mathbb T$, the distribution of $\Psi_x(\mathbb{T}_*^{\le x})$ and $\mathbb{T}_*^{\leq \hat x}$ are the same.
	\end{lemma}
 
We fix the tree $T$ and specify the weight of each edge on these trees. We denote a directed edge by $e:=(\underline e,\overline e)$ and the reversed one by $\check e:=(\overline{e},\underline{e})$. Also, let $\check\gamma:=Rev(\gamma)$ be the reverse of the path $\gamma$. By definition, $\Psi_x(e)=(\Psi_x(\underline{e}),\Psi_x(\overline{e}))$ for $\underline e\in T_*$. Define 
	\begin{equation*}
		R_x:=\{e:\rho_*\preceq \overline{e}\preceq x,\rho_*\prec\underline{e}\prec x\}
	\end{equation*}
 and 
 \begin{equation*}
     \alpha'_{\Psi_x(e)}:=\left\{\begin{array}{ll}
			\alpha_{\check{e}} & e\in R_x, \\
			\alpha_e & \text{otherwise}.
		\end{array}\right.
 \end{equation*}
When $x=\rho11$, for example, $R_{\rho11}=\{(\rho,\rho1),(\rho,\rho_*),(\rho1,\rho11),(\rho1,\rho)\}$.

  \begin{figure}[ht]
     \centering
     \includegraphics[width=0.75\linewidth]{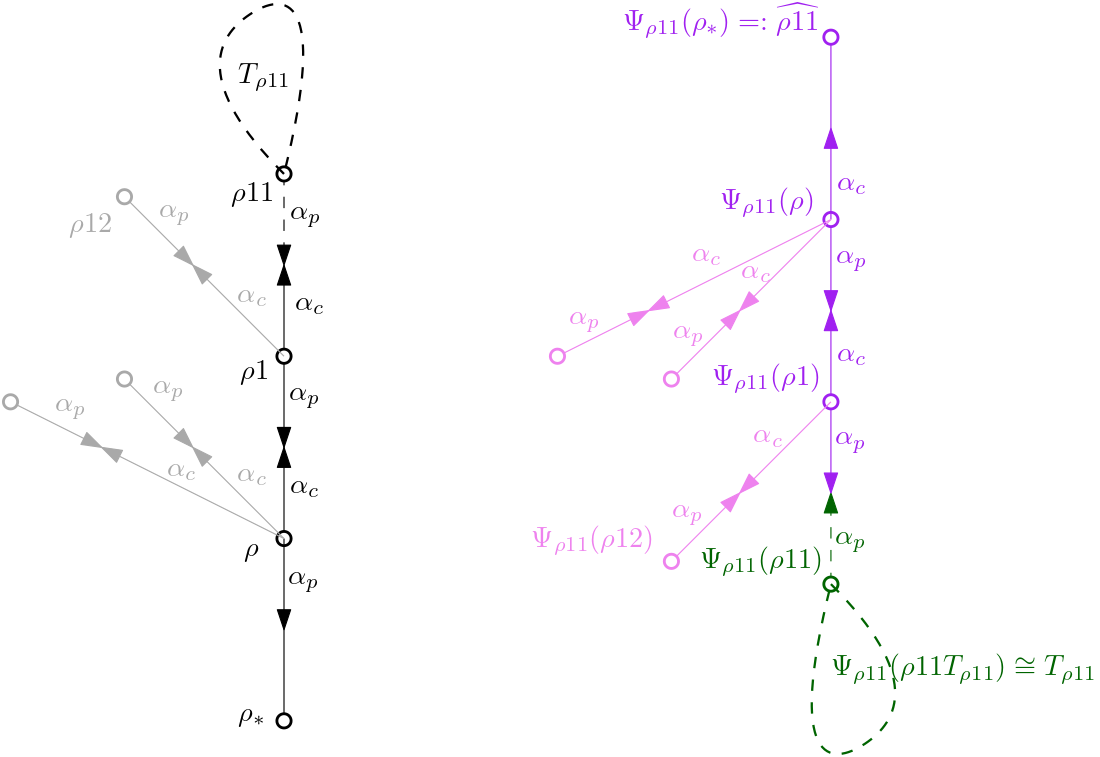}
     \caption{Trees $ T_*$ and $\Psi_{\rho11}( T_*)$ with edge weights on them. The purple part and green part of $\Psi_{\rho11}(T_*)$ correspond to $T^-$ and $T^+$ in the double tree respectively (recall Figure \ref{fig:double_tree}). Therefore, one can consider $\Psi$ as the map from a tree to a double tree.}
     \label{fig_map_psi}
 \end{figure}

We list here some basic facts about edge local times. 
 \begin{fact}\label{Fact_Psi}
For any fixed tree $T_*$, any vertex $x\in T$, and any path $\gamma=(y_i)_{0\leq i\leq n}$ on the tree such that $y_0=\rho_*,y_n=x$, we have
\begin{enumerate}
\item[(1)] for any edge $e$ such that $\underline e\in T_*$, $N_e(\check\gamma)=N_{\check e}(\gamma)$;
    \item[(2)]  for any edge $\underline e\in T_*$, $N_{\Psi_{x}(e)}(\Psi_x(\gamma))=N_e(\gamma)$, i.e. $\Psi_x$ does not change the edge local times;
    \item[(3)]  for any vertex $y\in T_*\backslash\{\rho_*,x\}$, $\sum_{\underline{e}=y}N_e(\gamma)=\sum_{\overline{e}=y}N_e(\gamma)$, i.e. inflow equals outflow except for the source and sink;
  \item[(4)]   for any edge $e\not\in R_{x}\cup\{(\rho_*,\rho),(x,x_*) \}$, $N_{e}(\gamma)=N_{\check e}(\gamma)$ since trees are acyclic.

\end{enumerate}
\end{fact}
 \begin{figure}[ht]
     \centering
     \includegraphics[width=0.75\linewidth]{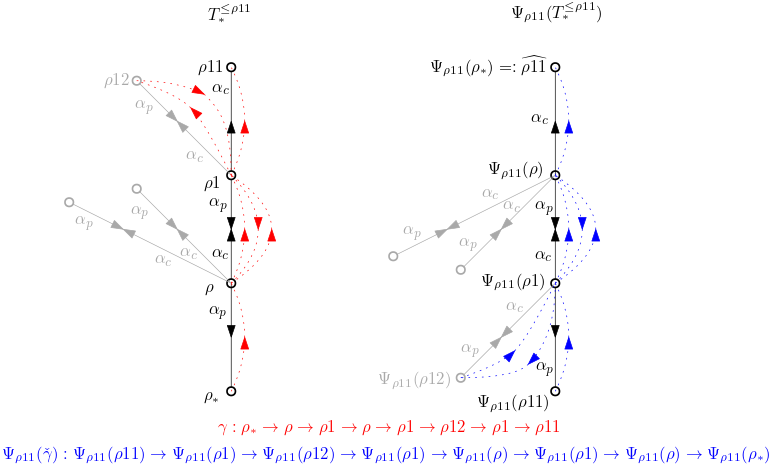}
     \caption{Left: $\gamma$ (red path) on $ T_*^{\leq \rho11}$. Right: $\Psi_{\rho11}(\check{\gamma})$ (blue path) on $\Psi_{\rho11}( T_*^{\leq \rho11})$. The black part of the tree is $R_x$, where we exchange the weights; the weights of the gray part stay the same. We use blue color to indicate that the path is reversed.}
 \end{figure}
 
 For any graph $G$ with edge weights $(\alpha_e)_{ e\in E}$, we denote by $\mathbf{P}^{G}(\cdot|\alpha)$ the law of the associated ERRW on $G$ for clarity. Sometimes we write $\mathbf P(\cdot|\alpha)$ (resp. $\mathbf{P}^{G}(\cdot)$) when there is no confusion about the graph (resp. weight). For a fixed path $\gamma=(y_i)_{0\le i\le n}$, we simplify $\mathbf P((X_i)_{0\le i\le n}=\gamma)$ as $\mathbf P(\gamma)$. For $\gamma=(y_i)_{i\ge 0}$, we define $\gamma|_{\ge 1}$ as $(y_i)_{i\ge 1}$. \begin{remark}
		In the following lemma, we use general $(\alpha_e)_e$ instead of  $(\alpha_p,\alpha_c)$ since we need the general setting for a later proof. As for the $(\alpha_p,\alpha_c)$-case, the ratio is just 1. It also explains why we do not consider the case when edges pointing towards offspring have different weights $\alpha_{c,1},\cdots,\alpha_{c,{\nu}}$.
\end{remark}
\begin{lemma}
		\label{lemma_fresh_rev}
		Fix a tree $T_*$ and consider ERRW on $T_*^{\leq x}$ and $\Psi_x(T_*^{\leq x})$. Take a finite path $\gamma=(y_i)_{0\leq i\leq n}$ such that $y_0=\rho_*,y_n=x,y_i\not\in\{\rho_*,x\},1\leq i\leq n-1$. Then we have
		\begin{equation}\label{eq-pathrev}
			\mathbf{P}^{T_*^{\leq x}}(\gamma|\alpha)=\mathbf{P}^{T_*^{\leq x}}(\gamma|_{\ge 1}|\alpha)=\frac{\alpha_{(x_*,x)}}{\alpha_{(\rho_*,\rho)}} \mathbf{P}^{\Psi_x(T_*^{\leq x})}(\Psi_x(\check{\gamma})|_{\ge1}|\alpha^\prime),
		\end{equation}
if $(\alpha_e)_{\underline e\in T_*^{\le x}}$ satisfies the condition that
\begin{align}\label{eq_condition_weight}
    \alpha_{(y,y_*)}+\alpha_{(y,yj)}= \alpha_{(y_*,y)}+\alpha_{(yj,y)},\text{ for all }y\text{ and } j\text{ s.t. } \rho\preceq y\prec yj\preceq x. 
\end{align}
\end{lemma}
\begin{proof}
We call $N_{\Psi_x(e)}(\Psi_x(\check{\gamma})|_{\ge1})$ the edge local time of $\Psi_x(e)$ w.r.t. path $\Psi_x(\check{\gamma})|_{\ge1}$. It is easy to check that equation (\ref{eq_condition_weight}) implies
\begin{equation}\label{eq_alpha_and_prime}
    \sum_{\underline{e}=y}\alpha_e=\sum_{\underline{e}=y}\alpha_{\Psi_x(e)}',
    \, y\in T\backslash \{x\},
\end{equation}
since $\alpha_{(y,y_*)}+\alpha_{(y,yj)}= \alpha_{(y_*,y)}+\alpha_{(yj,y)}=\alpha_{\Psi_x((y,y_*))}'+\alpha_{\Psi_x((y,yj))}'$ for $ \rho\preceq y\prec yj\preceq x$.
According to Lemma \ref{lemma_ERRWpath}, we have
		\begin{equation*}	
  \begin{split}
      \mathbf{P}^{T_*^{\leq x}}(\gamma|\alpha)&=\prod_{y\in T_*}\frac{\Gamma(\sum_{\underline{e}=y}\alpha_e)}{\Gamma(\sum_{\underline{e}=y}\alpha_e+N_e(\gamma))}\prod_{\underline{e}\in T_*}\frac{\Gamma(\alpha_e+N_e(\gamma))}{\Gamma(\alpha_e)}\\
      &=\prod_{y\in T^{< x}}\frac{\Gamma(\sum_{\underline{e}=y}\alpha_e)}{\Gamma(\sum_{\underline{e}=y}\alpha_e+N_e(\gamma))}\prod_{\underline{e}\in T^{< x}}\frac{\Gamma(\alpha_e+N_e(\gamma))}{\Gamma(\alpha_e)},
  \end{split}
		\end{equation*}
		and
		\begin{equation*}	
			\begin{split}
				&\mathbf{P}^{\Psi_x(T_*^{\le x})}(\Psi_x(\check{\gamma})|_{\ge1}|\alpha^\prime)\\
    =&\prod_{y\in \Psi_x(T_*)}\frac{\Gamma(\sum_{\underline{e}=y}\alpha'_e)}{\Gamma(\sum_{\underline{e}=y}\alpha'_e+N_e(\Psi_x(\check{\gamma})|_{\ge 1}))}\prod_{\underline e\in \Psi_x(T_*)}\frac{\Gamma(\alpha'_e+N_e(\Psi_x(\check{\gamma})|_{\ge 1}))}{\Gamma(\alpha'_e)}\\
    =&\prod_{y\in T_*}\frac{\Gamma(\sum_{\underline{e}=y}\alpha'_{\Psi_x(e)})}{\Gamma(\sum_{\underline{e}=y}\alpha'_{\Psi_x(e)}+N_{\Psi_x(e)}(\Psi_x(\check\gamma)|_{\ge 1}))}\prod_{\underline{e}\in T_*}\frac{\Gamma(\alpha'_{\Psi_x(e)}+N_{\Psi_x(e)}(\Psi_x(\check\gamma)|_{\ge 1}))}{\Gamma(\alpha'_{\Psi_x(e)})}\\
    =&\prod_{y\in T_*^{< x}}\frac{\Gamma(\sum_{\underline{e}=y}\alpha'_{\Psi_x(e)})}{\Gamma(\sum_{\underline{e}=y}\alpha'_{\Psi_x(e)}+N_{\Psi_x(e)}(\Psi_x(\check\gamma)))}\prod_{\underline{e}\in T_*^{<x}}\frac{\Gamma(\alpha'_{\Psi_x(e)}+N_{\Psi_x(e)}(\Psi_x(\check\gamma)))}{\Gamma(\alpha'_{\Psi_x(e)})},
			\end{split}
		\end{equation*}
where the second line is obtained by the definition of $\Psi_x$.

Since $N_{\Psi_x(e)}(\Psi_x(\check\gamma))=N_{e}(\check\gamma)=N_{\check{e}}(\gamma)$ for each edge $\underline e\in T_*$ from Fact \ref{Fact_Psi} (2) and (1),
\begin{align*}
    &\prod_{y\in T_*^{< x}}\frac{\Gamma(\sum_{\underline{e}=y}\alpha'_{\Psi_x(e)})}{\Gamma(\sum_{\underline{e}=y}\alpha'_{\Psi_x(e)}+N_{\Psi_x(e)}(\Psi_x(\check\gamma)))}\prod_{\underline{e}\in T_*^{<x}}\frac{\Gamma(\alpha'_{\Psi_x(e)}+N_{\Psi_x(e)}(\Psi_x(\check\gamma)))}{\Gamma(\alpha'_{\Psi_x(e)})}\\
    =&\prod_{y\in T_*^{< x}}\frac{\Gamma(\sum_{\underline{e}=y}\alpha'_{\Psi_x(e)})}{\Gamma(\sum_{\underline{e}=y}\alpha'_{\Psi_x(e)}+N_{\check e}(\gamma))}\prod_{\underline{e}\in T_*^{< x}}\frac{\Gamma(\alpha'_{\Psi_x(e)}+N_{\check e}(\gamma))}{\Gamma(\alpha'_{\Psi_x(e)})}\\
    =&\prod_{y\in T^{< x}}\frac{\Gamma(\sum_{\underline{e}=y}\alpha'_{\Psi_x(e)})}{\Gamma(\sum_{\underline{e}=y}\alpha'_{\Psi_x(e)}+N_{\check e}(\gamma))}\prod_{\underline{e}\in T_*^{< x}}\frac{\Gamma(\alpha'_{\Psi_x(e)}+N_{\check e}(\gamma))}{\Gamma(\alpha'_{\Psi_x(e)})}.
\end{align*}
The last line follows from the fact that $N_{(\rho,\rho_*)}(\gamma)=0$.

For $y\notin \{\rho_*,x\}$, it holds that $\sum_{\underline{e}=y}N_{\check e}(\gamma)=\sum_{\overline{e}=y}N_e(\gamma)=\sum_{\underline{e}=y}N_e(\gamma)$ by Fact \ref{Fact_Psi} (3). Therefore, together with (\ref{eq_alpha_and_prime}), we have
		\begin{equation*}
			\prod_{y\in T^{< x}}\frac{\Gamma(\sum_{\underline{e}=y}\alpha'_{\Psi_x(e)})}{\Gamma(\sum_{\underline{e}=y}\alpha'_{\Psi_x(e)}+N_{\check{e}}(\gamma))}	=\prod_{y\in T^{< x} }\frac{\Gamma(\sum_{\underline{e}=y}\alpha_e)}{\Gamma(\sum_{\underline{e}=y}\alpha_e+N_e(\gamma))}.
		\end{equation*}
We also have $N_e(\gamma)=N_{\check{e}}(\gamma),e\not\in R_x\cup \{(\rho_*,\rho),(x,x_*)\}$ (Fact \ref{Fact_Psi}(4)), and for $e\not\in R_x\cup \{(\rho_*,\rho),(x,x_*)\}$, $\alpha'_{\Psi_x(e)}=\alpha_e$ by definition. Hence,
		\begin{equation*}
			\prod_{e\not\in R_x\cup \{(\rho_*,\rho),(x,x_*)\}}\frac{\Gamma(\alpha'_{\Psi_x(e)}+N_{\check e}(\gamma))}{\Gamma(\alpha'_{\Psi_x(e)})}=\prod_{e\not\in R_x\cup \{(\rho_*,\rho),(x,x_*)\}}\frac{\Gamma(\alpha_e+N_e(\gamma))}{\Gamma(\alpha_e)}.
		\end{equation*}
Finally, for $e\in R_x\backslash\{(\rho,\rho_*),(x_*,x) \}$, we have $\alpha'_{\Psi_x(e)}=\alpha_{\check e}$. Thus,
		\begin{align*}
			\prod_{e\in R_x\backslash\{(\rho,\rho_*),(x_*,x)\}}\frac{\Gamma(\alpha'_{\Psi_x(e)}+N_{\check e}(\gamma))}{\Gamma(\alpha'_{\Psi_x(e)})}&=\prod_{e\in R_x\backslash\{(\rho,\rho_*),(x_*,x)\}}\frac{\Gamma(\alpha_{\check{e}}+N_{\check{e}}(\gamma))}{\Gamma(\alpha_{\check{e}})}\\
   &=\prod_{ e \in R_x\backslash\{(\rho,\rho_*),(x_*,x)\}}\frac{\Gamma(\alpha_e+N_e(\gamma))}{\Gamma(\alpha_e)}.
		\end{align*}
The path $\gamma$ never visits $\{\rho_*,x\}$ except for the start and end, which implies $N_{(\rho_*,\rho)}(\gamma)=N_{(x_*,x)}(\gamma)=1$. Till now, by canceling the corresponding terms, we only have
\begin{equation*}
\begin{split}
    \frac{\Gamma(\alpha_{(x_*,x)}+N_{(x_*,x)}(\gamma))}{\Gamma(\alpha_{(x_*,x)})}
    =\frac{\Gamma(\alpha_{(x_*,x)}+1)}{\Gamma(\alpha_{(x_*,x)})}=\alpha_{(x_*,x)}
\end{split}
\end{equation*}
left on the left-hand side of (\ref{eq-pathrev}). In the same way, the right-hand side of (\ref{eq-pathrev}) also remains
\begin{equation*}
\begin{split}
        \frac{\alpha_{(x_*,x)}}{\alpha_{(\rho_*,\rho)}}&\frac{\Gamma(\alpha'_{\Psi_x((\rho,\rho_*))}+N_{(\rho_*,\rho)}(\gamma))}{\Gamma(\alpha'_{\Psi_x((\rho,\rho_*))})}=\frac{\alpha_{(x_*,x)}}{\alpha_{(\rho_*,\rho)}}\frac{\Gamma(\alpha_{(\rho_*,\rho)}+1)}{\Gamma(\alpha_{(\rho_*,\rho)})} =\alpha_{(x_*,x)}.
\end{split}
\end{equation*}
The proof is complete
\end{proof}
As an immediate consequence of Lemma \ref{lemma_fresh_rev}, we present the distribution of the tree and path seen at a fresh point $X_{\theta_k}$ (see Section \ref{Sec_Fact_Reg} for the definition of $\theta_k$). Recall that we denote by  $\mathbb P(\cdot)$ the annealed distribution of the ERRW $(X_n)_{n\geq 0}$ on a Galton-Watson tree and we write $\mathbb P_x(\cdot):=\mathbb P(\cdot|X_0=x)$ for $x\in \mathbb T_*$ and $\mathbb E_x$ the corresponding expectation.
	\begin{corollary}\label{cor-rev-env}
		
		Let $\mathbb T$ be a Galton-Watson tree. For an $(\alpha_p,\alpha_c)$-ERRW $(X_n)_{n\ge 0}$, under $\mathbb{P}_{\rho_*}(\cdot|\theta_k<\tau_{\rho_*})$, we have
		\begin{equation*}
			\left(\Psi_{X_{\theta_k}}(\mathbb{T}^{\leq X_{\theta_k}}_*),\Psi_{X_{\theta_k}}((X_{\theta_k-j})_{0\le j\leq\theta_k})\right)\overset{(d)}{=}\left(\mathbb{T}^{\leq X_{\theta_k}}_*,(X_j)_{0\le j\leq\theta_k}\right).
		\end{equation*}
		In particular, $\hat{X}_{\theta_k}:=\Psi_{X_{\theta_k}}(\rho_*)$ follows the distribution of the $k$-th fresh point of an $(\alpha_p,\alpha_c)$-ERRW.     
	\end{corollary}

Next, we continue to investigate path reversibility for ERRW, when $\gamma$ does not necessarily stop at the first arrival of a point, i.e. $\gamma$ can be decomposed into a path first arriving at the point and several loops rooted at it. Recall that the function $\Phi$ was defined in \eqref{eq_phi_intro}.  
	\begin{lemma}
		\label{rev_XY}
		Fix a tree $T_*$ with $(\alpha_p,\alpha_c)$-weights. Let $\gamma_1$ be a path as in Lemma \ref{lemma_fresh_rev}, and $\gamma_2$ an arbitrary path that starts from $x$, ends at a vertex $y\in xT_x$,  and never visits $\rho_*$. We have
		\begin{equation}\label{eq_path_tree_to_Doubletree}
  \begin{split}
     & \mathbf{P}^{T_*}(\gamma_1|\alpha)\mathbf{P}^{T_*}(\gamma_2|\alpha+N(\gamma_1))\\
   =&\Phi(N_{(x_*,x)}(\gamma_2))\mathbf{P}^{\Psi_x(T_*)}(\Psi_x(\check \gamma_1)|_{\ge 1}|\alpha^\prime)\mathbf{P}^{\Psi_x(T_*) }(\Psi_x(\gamma_2)|\alpha^\prime+N(\Psi_x(\check \gamma_1)|_{\ge 1})).
  \end{split}
		\end{equation}
	\end{lemma}
 \begin{figure}[ht]
     \centering
     \includegraphics[width=0.75\linewidth]{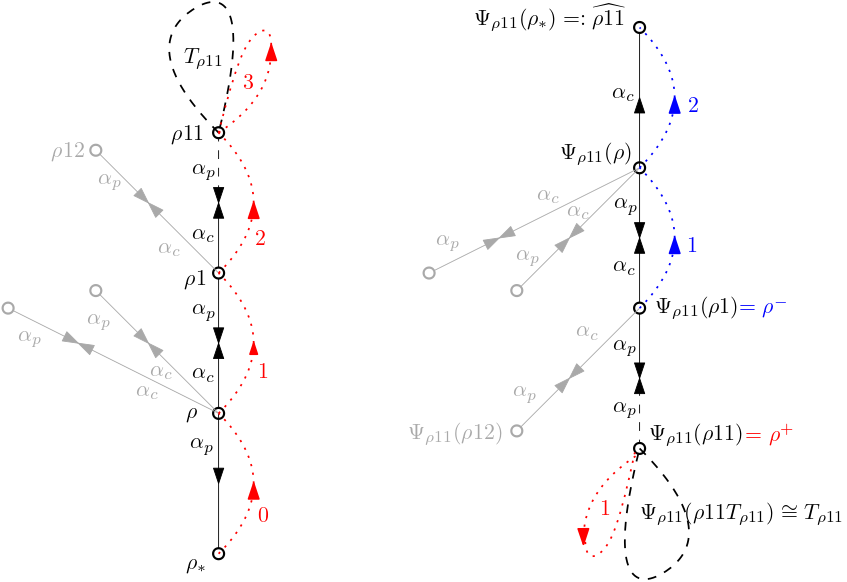}
     \caption{Left: the red path ends at $\rho11$ and we decompose it into $\gamma_1$ and $\gamma_2$ by 
the time $\tau_{\rho11}$. Right: the blue path is $\Psi_{\rho11}(\check \gamma_1)|_{\ge 1}$ and the red path is $\Psi_{\rho11}(\gamma_2)$. Recall that in Figure \ref{fig:XYonDoubleTree}, we use red color for $Y$ and blue color for $X$, which is consistent with the colors here.}
     \label{fig_two_paths}
 \end{figure}
	\begin{proof}
		We claim that we only need to prove
		\begin{equation}
			\label{eq_rev1}
			{\mathbf P}^{ T_*}(\gamma_2|\alpha)=\frac{\Gamma(\alpha_{(x,x_{*})})\Gamma(\alpha_{(x_*,x)}+N_{(x_*,x)}(\gamma_2))}{\Gamma(\alpha_{(x_*,x)})\Gamma(\alpha_{(x,x_{*})}+N_{(x_*,x)}(\gamma_2))}{\mathbf P}^{\Psi_x( T_*)}(\Psi_x(\gamma_2)|\alpha')
		\end{equation}
		and
		\begin{equation}
			\label{eq_rev2}
			{\mathbf P}^{ T_*}(\gamma_1|\alpha+N(\gamma_2))=\frac{\alpha_{(x_*,x)}+N_{(x_*,x)}(\gamma_2)}{\alpha_{(\rho_*,\rho)}+N_{(\rho_*,\rho)}(\gamma_2)}{\mathbf P}^{\Psi_x( T_*)}(\Psi_x(\check \gamma_1)|_{\ge 1}|\alpha'+N(\Psi_x(\gamma_2))),
		\end{equation}
		where $\alpha,\alpha'$ are weights of $ T_*,\Psi_x( T_*)$ respectively. In fact, for the $(\alpha_p,\alpha_c)$ case, $\alpha_{(x,x_*)}=\alpha_p$ and $\alpha_{(x_*,x)}=\alpha_{(\rho_*,\rho)}=\alpha_c$. Also, by the assumption on $\gamma_2$ we have $N_{(\rho_*,\rho)}(\gamma_2)=0$. Hence, by the formula $\Gamma(\alpha)\alpha=\Gamma(\alpha+1)$, it holds that
  \begin{align*}
      \frac{\Gamma(\alpha_{(x,x_{*})})\Gamma(\alpha_{(x_*,x)}+N_{(x_*,x)}(\gamma_2))}{\Gamma(\alpha_{(x_*,x)})\Gamma(\alpha_{(x,x_{*})}+N_{(x_*,x)}(\gamma_2))}\times\frac{\alpha_{(x_*,x)}+N_{(x_*,x)}(\gamma_2)}{\alpha_{(\rho_*,\rho)}+N_{(\rho_*,\rho)}(\gamma_2)}=\frac{\Gamma(\alpha_{c}+1+N_{(x_*,x)}(\gamma_2))\Gamma(\alpha_p)}{\Gamma(\alpha_{p}+N_{(x_*,x)}(\gamma_2))\Gamma(\alpha_c+1)}
  \end{align*}
  and the lemma follows.
		
		We first prove (\ref{eq_rev2}). Notice that $N_{\check{e}}(\gamma_2)=N_e(\gamma_2)$ for $e$ such that $\underline{e}\in  T^{< x}_*$ since $\gamma_2$ starts from $x$ and ends at $y\in  xT_x$. By Fact \ref{Fact_Psi}(2), we also have $N_{\Psi_x(e)}(\Psi_x(\gamma_2))=N_e(\gamma_2)$ and $N_{\Psi_x(\check{e})}(\Psi_x(\gamma_2))=N_{\check{e}}(\gamma_2)$. We replace the weights $\alpha$ by $\alpha+N(\gamma_2)$ and $\alpha^\prime$ by $\alpha^\prime+N(\Psi_x(\gamma_2))$ respectively in (\ref{eq-pathrev}). Since $N_{\check{e}}(\gamma_2)=N_e(\gamma_2)=N_{\Psi_x(e)}(\Psi_x(\gamma_2))=N_{\Psi_x(\check{e})}(\Psi_x(\gamma_2))$ for $e$ such that $\underline{e}\in  T^{< x}_*$, the new sets of weights, $(\alpha_e+N_e(\gamma_2))_{\underline e\in T_*}$ and $(\alpha_e'+N_{\Psi_x(e)}(\Psi_x(\gamma_2)))_{\underline e\in \Psi_x(T_*)}$, also satisfy the condition in Lemma \ref{lemma_fresh_rev}. Thus Lemma \ref{lemma_fresh_rev} implies (\ref{eq_rev2}).
  
		For (\ref{eq_rev1}), by Lemma \ref{lemma_ERRWpath}, we have
  \begin{equation}
			\label{eq_p_temp_1}
			\mathbf P^{ T_*}(\gamma_2|\alpha)=\prod_{y\in T_*}\frac{\Gamma(\sum_{\underline{e}=y}\alpha_e)}{\Gamma(\sum_{\underline{e}=y}\alpha_e+N_e(\gamma_2))}\prod_{\underline e\in T_*}\frac{\Gamma(\alpha_e+N_e(\gamma_2))}{\Gamma(\alpha_e)},
		\end{equation}
	and
\begin{equation}
\label{eq_p_temp_2}
			{\mathbf P}^{\Psi_x( T_*)}(\Psi_x(\gamma_2)|\alpha')=\prod_{y\in T^*}\frac{\Gamma(\sum_{\underline{e}=y}\alpha_{\Psi_x(e)}')}{\Gamma(\sum_{\underline{e}=y}\alpha_{\Psi_x(e)}'+N_e(\gamma_2))}\prod_{\underline e\in T_*}\frac{\Gamma(\alpha_{\Psi_x(e)}'+N_e(\gamma_2))}{\Gamma(\alpha_{\Psi_x(e)}')}.
		\end{equation}
Since for any vertex $y\in T_*$, $\sum_{\underline{e}=y}\alpha_e=\sum_{\underline{e}=\Psi_x(y)}\alpha_{\Psi_x(e)}'$, we have 
\begin{align*}
   \prod_{y\in T_*}\frac{\Gamma(\sum_{\underline{e}=y}\alpha_e)}{\Gamma(\sum_{\underline{e}=y}\alpha_e+N_e(\gamma_2))}= \prod_{y\in T^*}\frac{\Gamma(\sum_{\underline{e}=y}\alpha_{\Psi_x(e)}')}{\Gamma(\sum_{\underline{e}=y}\alpha_{\Psi_x(e)}'+N_e(\gamma_2))}.
\end{align*}

Note that $\alpha_e=\alpha'_{\Psi_x(e)}$ for $e$ such that $\underline e\in xT_x$. For $e\in R_x\backslash\{(x_*,x),(\rho,\rho_*)\}$, we have $\check{e}\in R_x\backslash\{(x_*,x),(\rho,\rho_*)\}$, $\alpha_e=\alpha^\prime_{\Psi_x(\check e)}$ by definition and $N_e(\gamma_2)=N_{\check{e}}(\gamma_2)$ by Fact \ref{Fact_Psi}(2). For $e\in \{e:\underline{e}\in T^{<x}\}\backslash R_x$, we have $\alpha_{e}=\alpha^\prime_{\Psi_x(e)}$. Repeating our discussion in Lemma \ref{lemma_fresh_rev}, the only different parts between (\ref{eq_p_temp_1}) and (\ref{eq_p_temp_2}) are the probabilities of going from $x_*$ to $x$ and from $\Psi_x(x_*)$ to $\Psi_x(x)$, which are 
		\begin{equation*}
			\frac{\Gamma(\alpha_{(x_*,x)}+N_{(x_*,x)}(\gamma_2))}{\Gamma(\alpha_{(x_*,x)})},
		\end{equation*}
		and
		\begin{equation*}
			\frac{\Gamma(\alpha_{\Psi_x((x_*,x))}'+N_{(x_*,x)}(\gamma_2))}{\Gamma(\alpha_{\Psi_x((x_*,x))}')}
		\end{equation*}
	respectively. Their ratio is 
 \begin{align*}
     \frac{\Gamma(\alpha_{(x,x_{*})})\Gamma(\alpha_{(x_*,x)}+N_{(x_*,x)}(\gamma_2))}{\Gamma(\alpha_{(x_*,x)})\Gamma(\alpha_{(x,x_{*})}+N_{(x_*,x)}(\gamma_2))}.
 \end{align*}

	\end{proof}
	\begin{remark}\label{RM_fixenv}
	The equation (\ref{eq_path_tree_to_Doubletree}) still holds when we condition on the environment $(\eta_e)_{\underline{e}\in x T_x}$ and $(\eta_e)_{\underline{e}\in \Psi_x(x T_x)}$ (which means that on $xT_x$ the walker follows the quenched law of a RWDE).	
 \end{remark}

 \subsection{Asymptotic distribution of the environment seen from a particle}
\label{section_asymptotic_dist}
The main result of this section is Theorem \ref{prop_asy_env} which implies that, under the condition \eqref{eq_assump_finite}, the asymptotic distribution of the environment seen from ERRW is the measure $\mu_{ER}$ defined in the introduction, renormalized to be a probability measure. We follow the strategy in \cite{aidekon2014speed} to establish this result.

We first state a lemma similar to \cite[ Lemma 4.4]{aidekon2014speed}. 
\begin{lemma}\label{Lm_exp_bais}
Let $Y$ be a RWDE on Galton-Watson double tree with $(\alpha_p,\alpha_c)$-environment, then
{\rm
    \begin{equation*}
        \begin{split}
             & {\rm E}_{\rho^+}^{\eta}\left[\sum_{l\ge 0}\Phi(N_{(\rho^-,\rho^+)}^{Y}(l))\textbf{1}_{\{Y_{l}=\rho^+\}}\right]\\
    =&\frac{(1-\beta(\rho^+))}{\eta(\rho^+,\rho^-)}\sum_{k\geq0}\Phi(k)(1-\beta(\rho^+))^k(1-\beta(\rho^-))^k.
        \end{split}
    \end{equation*}
    }
\end{lemma}
\begin{proof}
    Recall the definition of $\beta(x)$ in \eqref{eq_def_beta} and observe that
	\begin{align*}
		{\rm E}_{\rho^+}^\eta\left[\sum_{l\geq0}\Phi(N^{Y}_{(\rho^-,\rho^+)}(l))\textbf{1}_{\{Y_l=\rho^+\}}\right]=\sum_{k\geq0}\Phi(k)\sum_{l\geq 0} {\rm P}_{\rho^+}^\eta(Y_l=\rho^+,N_{(\rho^+,\rho^-)}^{Y}(l)={k}).
	\end{align*}
	We can define the stopping times 
	\begin{align*}
		s_{k}:=\inf \{l\geq 0: N_{(\rho^+,\rho^-)}^{Y}(l)={k}\}
	\end{align*}
and $t_{k}:=\inf \{l\geq s_{k}: Y_{l}=\rho^+ \}$. We have a formula for conductance obtained by Markov property that
	\begin{equation}\label{eq-conductance}
		\frac1{\beta(x)}=1+\frac{\eta(x,x_*)}{\sum_{i=1}^{\nu(x)}\eta(x,xi)\beta(xi)}.
	\end{equation}
	Then by the Markov property at $t_{k}$, we have 
	\begin{align*}
		&\sum_{l\geq 0} {\rm P}_{\rho^+}^\eta(Y_l=\rho^+,N_{(\rho^-,\rho^+)}^{Y}(l)={k})={\rm E}_{\rho^+}^\eta\left[\textbf {1}_{\{t_{k}<\infty\}} \sum_{l=t_{k}}^{s_{{k}+1}}\textbf {1}_{\{Y_l=\rho^+\}}\right]\\
		=&{\rm P}_{\rho^+}^\eta(t_{k}<\infty){\rm E}_{\rho^+}^\eta\left[\sum_{l=0}^{s_1}\textbf {1}_{\{Y_l=\rho^+\}}\right]=(1-\beta(\rho^-))^{k}(1-\beta(\rho^+))^{k} \frac{1}{1-{\rm P}_{\rho^+}^\eta(\tau_{\rho^+}<s_1)}\\
		=&(1-\beta(\rho^-))^{k}(1-\beta(\rho^+))^{k}\frac{1}{\eta(\rho^+,\rho^-)+\sum_{i=1}^{\nu(\rho^+)}\eta(\rho^+,\rho^+i)\beta(\rho^+i)}\\
  =&(1-\beta(\rho^-))^{k}(1-\beta(\rho^+))^{k}\frac{1-\beta(\rho^+)}{\eta(\rho^+,\rho^-)},
	\end{align*}
	where the last equality comes from the formula (\ref{eq-conductance}). The proof is complete.
\end{proof}

In order to describe the limit distribution, we construct the law for two dependent ERRW on a double Galton-Watson tree. We first sample a double Galton-Watson tree $\mathbb T^-\rightleftharpoons\mathbb T^+$ with weight $\alpha$ defined in \eqref{eq_weightofDT}. We let $(X_n)_{n\ge 0}$ (resp $(Y_n)_{n\ge 0}$) be an ERRW with law $\mathbf P^{\mathbb T^-\rightleftharpoons\mathbb T^+}(\cdot|\alpha)$ (resp. $\mathbf P^{\mathbb T^-\rightleftharpoons\mathbb T^+}(\cdot|\alpha+N^X)$) starting from $\rho^-$ (resp. $\rho^+$). In other words, we first run an ERRW $(X_n)_{n\ge 0}$, add the local time of $(X_n)_{n\ge 0}$ to the initial weight $\alpha$ and run another ERRW $(Y_n)_{n\ge 0}$. We call the probability measure constructed above $\mathbb P_{(\rho^-,\rho^+)}(\cdot)  $ and denote by $\mathbb E_{(\rho^-,\rho^+)}[\cdot] $ the corresponding expectation. 

Recall that the law of an ERRW can be viewed as the annealed law of a RWDE. From Lemma \ref{lemma_path_sym}, we can view $\mathbb P_{(\rho_*,\rho)}[\cdot]$ as the annealed law of two independent RWDE on a double Galton-Watson tree. Recall the definition of $\mathscr C$ in \eqref{eq_assump_finite}. By independence of $\mathbb{T}^+$ and $\mathbb{T}^-$ in a double tree, 
\begin{align*}\label{eq_assump_finite}
			\mathscr{C}&=\mathbb{E}_{(\rho^-,\rho^+)}\left[\frac{\beta(\rho^-)(1-\beta(\rho^+))}{\eta(\rho^+,\rho^-)}\sum_{k\geq0}\Phi(k)(1-\beta(\rho^+))^k(1-\beta(\rho^-))^k\right].
	\end{align*}
With Lemma \ref{Lm_exp_bais}, 
 \begin{equation}\label{eq_C_expression}
     \begin{split}
             \mathscr{C}
     =&\mathbb{E}_{(\rho^-,\rho^+)}\left[ \sum_{l\ge 0}\Phi(N_{(\rho^-,\rho^+)}^{Y}(l))\textbf{1}_{\{\tau^X_{\rho^+}=\infty\}}\textbf{1}_{\{Y_{l}=\rho^+\}}\right].
     \end{split}
 \end{equation}

We now describe the state space of the environment seen from an ERRW. We define $(\Omega_{ER},\mathcal{F}_{ER})$  as
\begin{equation*}
		\Omega_{ER}:=\{(T^-\rightleftharpoons T^+,\gamma):\gamma\in \Omega_{\gamma}\},\,\mathcal{F}_{ER}:=\sigma_{finite}((T^-\rightleftharpoons T^+,\gamma)),
	\end{equation*}
where $\sigma_{finite}$ is the sigma field generated by information of finite subtrees and finite subpaths. The measure $\mu_{ER}$ given in \eqref{eq_def_muER} is a measure on $(\Omega_{ER},\mathcal{F}_{ER})$. 

\begin{theorem}\label{prop_asy_env}
Suppose that $m\in (1,\infty)$, and  \eqref{eq_transience}, \eqref{eq_assump_finite} hold. Under $\mathbb P(\cdot|\mathcal{S}) $, the random variable $(\Psi_{X_n}(\mathbb T_*^{< X_n})\rightleftharpoons \mathbb T_{X_n},(\eta_e)_{\underline e\in \mathbb T_{X_n}},\Psi_{X_n}((X_l)_{0\le l\le n}))$ converges in distribution as $n\rightarrow \infty$. More precisely, for any bounded functional $\mathbf{F}$,
{\rm
\begin{equation}\label{eq_Thm_asy_1}
    \begin{split}
&\lim_{n\rightarrow\infty} \mathbb E_{\rho}\left[ \mathbf F(\Psi_{X_n}(\mathbb T_*^{< X_n})\rightleftharpoons \mathbb T_{X_n},(\eta_e)_{\underline e\in \mathbb T_{X_n}},\Psi_{X_n}((X_l)_{0\le l\le n}) )|\mathcal{S}\right] =\mathscr C^{-1}\times \\
   &\mathbb E_{(\rho_-,\rho^+)}\left[\sum_{t\ge 0}\mathbf F(\mathbb T^-\rightleftharpoons\mathbb T^+, (\eta_e)_{\underline e\in \mathbb T^+},Rev((X_l)_{l\ge 0})\ast(Y_l)_{0\leq l\leq t} ) \Phi(N_{(\rho^-,\rho^+)}^{Y}(t))\textbf{1}_{\{\tau^X_{\rho^+}=\infty\}}\textbf{1}_{\{Y_{t}=\rho^+\}}\right], 
    \end{split}
\end{equation}
}
where $\mathscr{C}\in (0,\infty)$ is the renormalising constant defined in \eqref{eq_assump_finite} and it also equals
{\rm
\begin{align}\label{eq_Thm_asy_2}
    \mathbb E_{\rho}[\Theta_1\textbf{1}_{\{\tau_{\rho_*}=\infty\}}]\mathbb{E}_{\rho}\left[\beta(\rho)\right].
\end{align}
}
\end{theorem}
By averaging over the environment $\eta$, we deduce the following corollary, which implies that $\mathscr C^{-1}\mu_{ER}$ is the limiting  measure for the environment seen from an ERRW on a Galton-Watson tree. It will be proved in a subsequent paper that $\mu_{ER}$ is an invariant measure even when \eqref{eq_assump_finite} fails to hold.

\begin{corollary}\label{Cor_inv_meas}
    Suppose that $m\in (1,\infty)$, and  \eqref{eq_transience}, \eqref{eq_assump_finite} hold. Under $\mathbb P(\cdot|\mathcal{S}) $, the random variable $(\Psi_{X_n}(\mathbb T_*^{< X_n})\rightleftharpoons \mathbb T_{X_n},\Psi_{X_n}((X_l)_{0\le l\le n}))$ converges in distribution as $n\rightarrow \infty$ to $\mathscr{C}^{-1}\mu_{ER}$. 
\end{corollary}

\begin{proof}[Proof of Theorem \ref{prop_asy_env}]
Let $\mathbf F$ be a bounded measurable function on the space of marked trees. We suppose that there exists a constant $M>0$ such that $\mathbf F$ only depends on the subtree $\{u\in \mathbb T^-\rightleftharpoons\mathbb T^+,|u|\le M \}$ (and the part of path on it), where $|u|$ is the graph distance between $u$ and $\rho^+$. We need to show that 
\begin{align*}
    &\lim_{n\rightarrow\infty} \mathbb E_{\rho}\left[ \mathbf F(\Psi_{X_n}(\mathbb T_*^{< X_n})\rightleftharpoons \mathbb T_{X_n},(\eta_e)_{\underline e\in \mathbb T_{X_n}},\Psi_{X_n}((X_l)_{0\le l\le n}) )\textbf{1}_{\mathcal{S}} \right] =\frac{\mathbb P(\mathcal{S}) }{\mathbb E_{\rho}[\Theta_1\textbf{1}_{\{\tau_{\rho_*}=\infty\}}]\mathbb{E}_{\rho}\left[\beta(\rho)\right]}\times\\
    &\mathbb E_{(\rho_-,\rho^+)}\left[\sum_{t\ge 0}\mathbf F(\mathbb T^-\rightleftharpoons\mathbb T^+, (\eta_e)_{\underline e\in \mathbb T^+},Rev((X_l)_{l\ge 0})\ast(Y_l)_{0\leq l\leq t} ) \Phi(N_{(\rho^-,\rho^+)}^{Y}(t))\textbf{1}_{\{\tau^X_{\rho^+}=\infty\}}\textbf{1}_{\{Y_{t}=\rho^+\}}\right],
\end{align*}
to prove \eqref{eq_Thm_asy_1} and \eqref{eq_Thm_asy_2}.

 By the argument in \cite[Section 4.3]{aidekon2014speed}, it suffices to prove
\begin{equation}\label{eq_prop_asy_env_1}
    \begin{split}
           &\lim_{n\rightarrow\infty} \mathbb E_{\rho}\left[\mathbf F(\Psi_{X_n}(\mathbb T_*^{< X_n})\rightleftharpoons \mathbb T_{X_n},(\eta_e)_{\underline e\in \mathbb T_{X_n}},\Psi_{X_n}((X_l)_{0\le l\le n}) )\textbf{1}_{\{\tau_{\rho_*}>n\}} \right] =\frac{1}{\mathbb E_{\rho}[\Theta_1\textbf{1}_{\{\tau_{\rho_*}=\infty\}}]}\times\\
    &\mathbb E_{(\rho_-,\rho^+)}\left[\sum_{t\ge 0}\mathbf F(\mathbb T^-\rightleftharpoons\mathbb T^+, (\eta_e)_{\underline e\in \mathbb T^+},Rev((X_l)_{l\ge 0})\ast(Y_l)_{0\leq l\leq t} ) \Phi(N_{(\rho^-,\rho^+)}^{Y}(t))\textbf{1}_{\{\tau^X_{\rho^+}=\infty\}}\textbf{1}_{\{Y_{t}=\rho^+\}}\right].
    \end{split}
\end{equation}
 For a random tree $\mathbb T$,  we let $\mathcal{S}_{\mathbb T}$ be the event that $\mathbb T$ is infinite. By dominated convergence we have
\begin{align*}
 &\mathbb E_{\rho}\left[\mathbf F(\Psi_{X_n}(\mathbb T_*^{< X_n})\rightleftharpoons \mathbb T_{X_n},(\eta_e)_{\underline e\in \mathbb T_{X_n}},\Psi_{X_n}((X_l)_{0\le l\le n}))\textbf{1}_{\{\tau_{\rho_*}>n\}} \right]\\
 =& \mathbb E_{\rho}\left[\mathbf F(\Psi_{X_n}(\mathbb T_*^{< X_n})\rightleftharpoons \mathbb T_{X_n},(\eta_e)_{\underline e\in \mathbb T_{X_n}},\Psi_{X_n}((X_l)_{0\le l\le n}))\textbf{1}_{\{\tau_{\rho_*}>n,|X_n|\ge M\}} \textbf{1}_{\mathcal{S}_{{\mathbb{T}}_*^{\leq X_{n}}}} \right]+o_n(1).
\end{align*}
Recall that $\theta_k$ is the $k$-th fresh epoch and $\xi_k=X_{\theta_k}$. We see for $n\ge 0$,
\begin{equation}\label{eq_fresh_split}
    \begin{split}
        	&\mathbb{E}_{\rho}\left[\mathbf F(\Psi_{X_n}(\mathbb T_*^{< X_n})\rightleftharpoons \mathbb T_{X_n},(\eta_e)_{\underline e\in \mathbb T_{X_n}},\Psi_{X_n}((X_l)_{0\le l\le n}))\textbf{1}_{\{\tau_{\rho_*}>n,|X_{n}|>M\}}\textbf{1}_{\mathcal{S}_{\mathbb T_*^{\leq X_{n}}}}\right]\\
	=&\sum_{k=0}^{\infty}\mathbb{E}_{\rho}\left[\mathbf F(\Psi_{\xi_k}(\mathbb T_*^{< \xi_k})\rightleftharpoons \mathbb T_{\xi_k},(\eta_e)_{\underline e\in \mathbb T_{\xi_k}},\Psi_{\xi_k}((X_l)_{0\le l\le n}))\textbf{1}_{\{\tau_{\rho_*}>n,|\xi_k|>M\}}\textbf{1}_{\{X_n=\xi_k\}}\textbf{1}_{\mathcal{S}_{\mathbb T_*^{\leq\xi_k}}}\right].
    \end{split}
\end{equation}
By Lemma \ref{lemma_backwardpsi} and Lemma \ref{rev_XY}, we have for each $k\ge 0$
\begin{equation}\label{eq_reverse_in_proof_asy_env}
			\begin{split}
				&\mathbb{E}_{\rho}\left[\mathbf F(\Psi_{\xi_k}(\mathbb T_*^{< \xi_k})\rightleftharpoons \mathbb T_{\xi_k},(\eta_e)_{\underline e\in \mathbb T_{\xi_k}},\Psi_{\xi_k}((X_l)_{0\le l\le n})) \textbf{1}_{\{\tau_{\rho_*}>n,|\xi_k|>M\}}\textbf{1}_{\{X_n=\xi_k\}}\textbf{1}_{\mathcal{S}_{\mathbb T_*^{\leq \xi_k}}}\right]\\
				=&\mathbb{E}_{(\rho^-,\rho^+)}\bigg[\mathbf F({(\mathbb{T}}^-)^{\leq \xi_k}\rightleftharpoons\mathbb T^+,(\eta_e)_{\underline e\in\mathbb T^+},Rev((X_l)_{0\le l\le \theta_k })\ast(Y_l)_{0\leq l\leq n-\theta_k})\\
    &\Phi(N_{(\rho^-,\rho^+)}^{Y}(n-\theta_k))\textbf{1}_{\{\tau_{\rho^+}^X>\theta_k,|\xi_k|>M\}}
    \textbf{1}_{\{\tau_{\xi_k}^{ Y}>n-\theta_k, Y_{n-\theta_k}=\rho^+\}}\textbf{1}_{\mathcal{S}_{({\mathbb{T}^-)}_*^{\leq \xi_k}}}\bigg].\\
			\end{split}
		\end{equation}
Here we reverse the tree and path conditioned on $(\eta_e)_{\underline e\in \mathbb T_{\xi_k}}$ by Remark \ref{RM_fixenv}. 
  
Reasoning on the value of $n-\theta_k$ in (\ref{eq_reverse_in_proof_asy_env}), we see that 
\begin{equation*}
\begin{split}
     &\sum_{k=0}^{\infty}\Phi(N_{(\rho^-,\rho^+)}^{Y}(n-\theta_k))\textbf{1}_{\{\tau_{\rho^+}^X>\theta_k,|\xi_k|>M\}}\textbf{1}_{\{\tau_{\xi_k}^{ Y}>n-\theta_k, Y_{n-\theta_k}=\rho^+\}}\textbf{1}_{\mathcal{S}_{{(\mathbb{T}}_*^-)^{\leq \xi_k}}}\\
=&\sum_{t=0}^{n}\Phi(N_{(\rho^-,\rho^+)}^{Y}(t))\textbf{1}_{\{\tau_{\rho^+}^X>n-t,|X_{n-t}|>M,n-t\in \{\theta_k,k\ge 0\}\}}\textbf{1}_{\{\tau_{X_{n-t}}^{Y}>t, Y_{t}=\rho^+\}}\textbf{1}_{\mathcal{S}_{{(\mathbb{T}}_*^-)^{\leq X_{n-t}}}}.
\end{split}
\end{equation*}
From the definition of the regeneration epoch, 
\begin{equation}\label{eq_X_independent}
    \begin{split}
           &\textbf{1}_{\{\tau_{\rho^+}^X>n-t,|X_{n-t}|>M,n-t\in \{\Theta_k,k\ge 0\}\}}\textbf{1}_{\{\tau_{X_{n-t}}^{Y}>t, Y_{t}=\rho^+\}}\textbf{1}_{\mathcal{S}_{{(\mathbb{T}}_*^-)^{\leq X_{n-t}}}}\\
   =&\textbf{1}_{\{\tau_{\rho^+}^X>n-t,|X_{n-t}|>M,n-t\in \{\theta_k,k\ge 0\}\}}\textbf{1}_{\{X_s\neq (X_{n-t})_*,s\ge n-t  \}}\textbf{1}_{\{\tau_{X_{n-t}}^{Y}>t, Y_{t}=\rho^+\}}\textbf{1}_{\mathcal{S}_{{(\mathbb{T}}_*^-)^{\leq X_{n-t}}}}.\\
    \end{split}
\end{equation}
We apply the Markov property at time $n-t$ for $(X_n)_{n\ge 0}$ considered as a RWDE and the branching property of a Galton-Watson tree at $X_{n-t}$. The term $\textbf{1}_{\{X_s\neq (X_{n-t})_*,s\ge n-t  \}}$ is independent of the other terms (it is independent of $\mathbf F$ since $|X_{n-t}|>M$), with expectation $\mathbb E_{\rho}[\beta(\rho)]$. Thus 
\begin{equation}\label{eq_reverse_times_beta}
\begin{split}
& \mathbb E_{(\rho^-,\rho^+)}\left[\mathbf F(\cdot)\Phi(N_{(\rho^-,\rho^+)}^{Y}(t))\textbf{1}_{\{\tau_{\rho^+}^X>n-t,|X_{n-t}|>M,n-t\in \{\Theta_k,k\ge 0\}\}}\textbf{1}_{\{\tau_{X_{n-t}}^{Y}>t, Y_{t}=\rho^+\}}\textbf{1}_{\mathcal{S}_{{(\mathbb{T}}_*^-)^{\leq X_{n-t}}}}\right]\\
=& \mathbb E_{(\rho^-,\rho^+)}\left[\mathbf F(\cdot)\Phi(N_{(\rho^-,\rho^+)}^{Y}(t))\textbf{1}_{\{\tau_{\rho^+}^X>n-t,|X_{n-t}|>M,n-t\in \{\theta_k,k\ge 0\}\}}\textbf{1}_{\{\tau_{X_{n-t}}^{Y}>t, Y_{t}=\rho^+\}}\textbf{1}_{\mathcal{S}_{{(\mathbb{T}}_*^-)^{\leq X_{n-t}}}} \right]\times\\ &\mathbb E_{\rho}[\beta(\rho)].
\end{split}
\end{equation}
We also obtain from \eqref{eq_C_expression} and our assumption \eqref{eq_assump_finite} that
\begin{equation}\label{eq_domin}
    \begin{split}
          & \mathbb E_{(\rho^-,\rho^+)}\left[\sum_{t=0}^{n}\Phi(N_{(\rho^-,\rho^+)}^{Y}(t))\textbf{1}_{\{\tau_{\rho^+}^X>n-t,|X_{n-t}|>M,n-t\in \{\Theta_k,k\ge 0\}\}}\textbf{1}_{\{\tau_{X_{n-t}}^{Y}>t, Y_{t}=\rho^+\}}\textbf{1}_{\mathcal{S}_{{\mathbb{T}}_*^-}}\right]\\
    \le &\mathbb E_{(\rho^-,\rho^+)}\left[\sum_{t= 0}^{\infty}\Phi(N_{(\rho^-,\rho^+)}^{Y}(t))\textbf{1}_{\{\tau_{\rho^+}^X=\infty\}}\textbf{1}_{\{Y_{t}=\rho^+\}}\right]=\mathscr C<\infty
    \end{split}
\end{equation}
By \eqref{eq_domin}, with the dominating function
 \begin{align*}
     ||\mathbf F||_{\infty}\sum_{t= 0}^{\infty}\Phi(N_{(\rho^-,\rho^+)}^{Y}(t))\textbf{1}_{\{\tau^X_{\rho^+}=\infty\}}\textbf{1}_{\{Y_{t}=\rho^+\}},
 \end{align*}
we can apply dominated convergence to see that
\begin{align*}
     &\sum_{t=0}^{n}\mathbb E_{(\rho^-,\rho^+)}\Bigg[\mathbf F({(\mathbb{T}}^-)^{\leq X_{n-t}}\rightleftharpoons\mathbb T^+,(\eta_e)_{\underline e\in\mathbb T^+},Rev((X_l)_{0\le l\le n-t })\ast(Y_l)_{0\leq l\leq t})\\
     &\Phi(N_{(\rho^-,\rho^+)}^{Y}(t))\textbf{1}_{\{\tau_{\rho^+}^X>n-t,|X_{n-t}|>M,n-t\in \{\Theta_k,k\ge 0\}\}}\textbf{1}_{\{\tau_{X_{n-t}}^{Y}>t, Y_{t}=\rho^+\}}\textbf{1}_{\mathcal{S}_{{(\mathbb{T}}_*^-)^{\leq X_{n-t}}}}\Bigg]\\
     =&\mathbb E_{(\rho^-,\rho^+)}\Bigg[\sum_{t=0}^{\infty}\mathbf F({\mathbb{T}}^-\rightleftharpoons\mathbb T^+,(\eta_e)_{\underline e\in\mathbb T^+},Rev((X_l)_{0\le l\le \infty })\ast(Y_l)_{0\leq l\leq t})\\
     &\Phi(N_{(\rho^-,\rho^+)}^{Y}(t))\textbf{1}_{\{\tau_{\rho^+}^X=\infty,n-t\in \{\Theta_k,k\ge 0\}\}}\textbf{1}_{\{\tau_{X_{n-t}}^{Y}>t, Y_{t}=\rho^+\}}\textbf{1}_{\mathcal{S}_{{\mathbb{T}}_*^-}}\Bigg]+o_n(1).
\end{align*}
Here we replace $\mathbf F({(\mathbb{T}}^-)^{\leq X_{n-t}}\rightleftharpoons\mathbb T^+,\cdot,\cdot)$ by $\mathbf F({\mathbb{T}}^-\rightleftharpoons\mathbb T^+,\cdot,\cdot)$, since $\mathbf F$ only relies on finite subtrees. Note that we can replace the quantity above by
\begin{equation*}
\begin{split}
    &\mathbb E_{(\rho^-,\rho^+)}\Bigg[\sum_{t=0}^{\infty}\mathbf F({\mathbb{T}}^-\rightleftharpoons\mathbb T^+,(\eta_e)_{\underline e\in\mathbb T^+},Rev((X_l)_{0\le l\le \infty })\ast(Y_l)_{0\leq l\leq t})\\
     &\Phi(N_{(\rho^-,\rho^+)}^{Y}(t))\textbf{1}_{\{\tau_{\rho^+}^X=\infty,n-t\in \{\Theta_k,k\ge t+M\}\}}\textbf{1}_{\{\tau_{X_{n-t}}^{Y}>t, Y_{t}=\rho^+\}}\textbf{1}_{\mathcal{S}_{{\mathbb{T}}_*^-}}\Bigg]+o_n(1)
\end{split}
\end{equation*}
by dominated convergence. Define $b_i:={\mathbb P}_{\rho}(i\in \{\Theta_k,k\ge0\}|\tau_{\rho_*}=\infty)$. According to the renewal theorem,
\begin{equation}\label{eq_conv_b_i}
    \lim_{i\to\infty}b_i=\frac{1}{\mathbb E_{\rho}[\Theta_1|\tau_{\rho_*}=\infty]}
\end{equation}
since $(\Theta_{k+1}-\Theta_{k})_{k\ge0}$ is an i.i.d. sequence when conditioned on $\{\tau_{\rho_*}=\infty\}$. Thus we have
\begin{equation*}
\begin{split}
    &\mathbb E_{(\rho^-,\rho^+)}\Bigg[\sum_{t=0}^{\infty}\mathbf F({\mathbb{T}}^-\rightleftharpoons\mathbb T^+,(\eta_e)_{\underline e\in\mathbb T^+},Rev((X_l)_{0\le l\le \infty })\ast(Y_l)_{0\leq l\leq t})\\
     &\Phi(N_{(\rho^-,\rho^+)}^{Y}(t))\textbf{1}_{\{\tau_{\rho^+}^X=\infty,n-t\in \{\Theta_k,k\ge t+M\}\}}\textbf{1}_{\{\tau_{X_{n-t}}^{Y}>t, Y_{t}=\rho^+\}}\textbf{1}_{\mathcal{S}_{{\mathbb{T}}_*^-}}\Bigg]\\
     =&\mathbb E_{(\rho^-,\rho^+)}\Bigg[\sum_{t=0}^{\infty}\mathbf F({\mathbb{T}}^-\rightleftharpoons\mathbb T^+,(\eta_e)_{\underline e\in\mathbb T^+},Rev((X_l)_{0\le l\le \infty })\ast(Y_l)_{0\leq l\leq t})\\
     &\Phi(N_{(\rho^-,\rho^+)}^{Y}(t))\textbf{1}_{\{\tau_{\rho^+}^X=\infty,n-t\ge\Theta_{t+M}\}}b_{n-t-\Theta_{t+M}}\textbf{1}_{\{\tau_{X_{n-t}}^{Y}>t, Y_{t}=\rho^+\}}\textbf{1}_{\mathcal{S}_{{\mathbb{T}}_*^-}}\Bigg]\\
     =&\frac{1}{\mathbb E_{\rho}[\Theta_1|\tau_{\rho_*}=\infty]}\mathbb E_{(\rho^-,\rho^+)}\Bigg[\sum_{t=0}^{\infty}\mathbf F({\mathbb{T}}^-\rightleftharpoons\mathbb T^+,(\eta_e)_{\underline e\in\mathbb T^+},Rev((X_l)_{0\le l\le \infty })\ast(Y_l)_{0\leq l\leq t})\\
     &\Phi(N_{(\rho^-,\rho^+)}^{Y}(t))\textbf{1}_{\{\tau_{\rho^+}^X=\infty\}}\textbf{1}_{\{Y_{t}=\rho^+\}}\textbf{1}_{\mathcal{S}_{{\mathbb{T}}_*^-}}\Bigg]+o_n(1),
\end{split}
\end{equation*}
where we use regeneration structure, branching property for the first equality, and \eqref{eq_conv_b_i} for the second. With \eqref{eq_reverse_times_beta}, we conclude that
\begin{align*}
    &\sum_{t=0}^{n}\mathbb E_{(\rho^-,\rho^+)}\Bigg[\mathbf F({(\mathbb{T}}^-)^{\leq X_{n-t}}\rightleftharpoons\mathbb T^+,(\eta_e)_{\underline e\in\mathbb T^+},Rev((X_l)_{0\le l\le n-t })\ast(Y_l)_{0\leq l\leq t})\\
     &\Phi(N_{(\rho^-,\rho^+)}^{Y}(t))\textbf{1}_{\{\tau_{\rho^+}^X>n-t,|X_{n-t}|>M,n-t\in \{\theta_k,k\ge 0\}\}}\textbf{1}_{\{\tau_{X_{n-t}}^{Y}>t, Y_{t}=\rho^+\}}\textbf{1}_{\mathcal{S}_{{(\mathbb{T}}_*^-)^{\leq X_{n-t}}}}\Bigg]\\
       =&\frac{1}{\mathbb E_{\rho}[\Theta_1\textbf{1}_{\{\tau_{\rho_*}=\infty\}}]}\mathbb E_{(\rho^-,\rho^+)}\Bigg[\sum_{t=0}^{\infty}\mathbf F({\mathbb{T}}^-\rightleftharpoons\mathbb T^+,(\eta_e)_{\underline e\in\mathbb T^+},Rev((X_l)_{0\le l\le \infty })\ast(Y_l)_{0\leq l\leq t})\\
     &\Phi(N_{(\rho^-,\rho^+)}^{Y}(t))\textbf{1}_{\{\tau_{\rho^+}^X=\infty\}}\textbf{1}_{\{Y_{t}=\rho^+\}}\textbf{1}_{\mathcal{S}_{{\mathbb{T}}_*^-}}\Bigg]+o_n(1).
\end{align*}
We thus complete the proof of \eqref{eq_prop_asy_env_1} and the theorem.
\end{proof}
We now show the equivalence between positiveness of the speed and finiteness of $\mathscr C$.

\begin{proof}[Proof of Theorem \ref{positive speed Thm 2}]
By dominated convergence, 
  \begin{align}\label{eq_finite_positive_0}
     0<\lim_{n\rightarrow\infty}\mathbb{E}_{\rho}\left[\textbf{1}_{\{\tau_{\rho_*}>n\}}\right]=\mathbb{E}_{\rho}\left[\beta(\rho)\right]<1
  \end{align}
  since the walk is transient. According to \eqref{eq_fresh_split}, \eqref{eq_reverse_in_proof_asy_env} and \eqref{eq_reverse_times_beta}, with $F\equiv1$ and $M=0$ (we do not use the condition \eqref{eq_assump_finite} and can drop the indicator function $\textbf{1}_{\mathcal{S}_{\mathbb T_*^{\leq X_{n}}}}$), we have
\begin{equation}\label{eq_finite_positive_2}
\begin{split}
&\mathbb{E}_{\rho}\left[\textbf{1}_{\{\tau_{\rho_*}>n\}}\right] \mathbb E_\rho[\beta(\rho)]\\
=& \mathbb E_{(\rho^-,\rho^+)}\left[\sum_{l=0}^{n}\Phi(N_{(\rho^-,\rho^+)}^{Y}(l))\textbf{1}_{\{\tau_{\rho^+}^X>n-l,n-l\in \{\theta_k,k\ge 0\}\}}\textbf{1}_{\{\tau_{X_{n-l}}^{Y}>l, Y_{l}=\rho^+\}}\right] \mathbb E_\rho [\beta(\rho)]\\
=& \mathbb E_{(\rho^-,\rho^+)}\left[\sum_{l=0}^{n}\Phi(N_{(\rho^-,\rho^+)}^{Y}(l))\textbf{1}_{\{\tau_{\rho^+}^X=\infty,n-l\in \{\Theta_k,k\ge 0\}\}}\textbf{1}_{\{\tau_{X_{n-l}}^{Y}>l, Y_{l}=\rho^+\}}\right].
\end{split}
\end{equation}

First, suppose the transient ERRW on a Galton-Watson tree $(X_n)_{n\ge 0}$ has a positive speed. From Fact \ref{lemma_regspeed}, we have $\mathbb{E}_{\rho}[\Theta_{1}|\tau_{\rho_*}=\infty]<\infty$. Fix $K\ge1$. We see that 
\begin{align*}
      &\textbf{1}_{\{\tau_{\rho^+}^X=\infty,n-l\in \{\Theta_k,k\ge 0\}\}}\textbf{1}_{\{\tau_{X_{n-l}}^{Y}>l, Y_{l}=\rho^+\}}\\
    \ge &\textbf{1}_{\{\tau_{\rho^+}^X=\infty,n-l\in \{\Theta_k,k\ge K\}\}}\textbf{1}_{\{\tau_{X_{\Theta_K}}^{Y}>l, Y_{l}=\rho^+\}}.
\end{align*}
Thus, we can apply the renewal theorem when $n-l$ goes to infinity. Recall that $b_i:={\mathbb P}_{\rho}(i\in \{\Theta_k,k\ge0\}|\tau_{\rho_*}=\infty)$. From the regeneration structure and branching property at $\Theta_K$,
\begin{equation}\label{eq_finite_positive_3}
\begin{split}
&\mathbb E_{(\rho^-,\rho^+)}\left[\sum_{l=0}^{n}\Phi(N_{(\rho^-,\rho^+)}^{Y}(l))\textbf{1}_{\{\tau_{\rho^+}^X=\infty,n-l\in \{\Theta_k,k\ge 0\}\}}\textbf{1}_{\{\tau_{X_{n-l}}^{Y}>l, Y_{l}=\rho^+\}}\right]\\
\ge &\mathbb{E}_{(\rho^-,\rho^+)}\left[\sum_{l=0}^{n}\Phi(N_{(\rho^-,\rho^+)}^{Y}(l))\textbf{1}_{\{\tau_{\rho^+}^X=\infty,n-l\in \{\Theta_k,k\ge K\}\}}\textbf{1}_{\{\tau_{X_{\Theta_K}}^{Y}>l, Y_{l}=\rho^+\}}\right]\\
    \ge&\mathbb{E}_{(\rho^-,\rho^+)}\left[\sum_{l=0}^{n/2}\Phi(N_{(\rho^-,\rho^+)}^{Y}(l))\textbf{1}_{\{\tau_{\rho^+}^X=\infty,n-l\ge\Theta_K\}}b_{n-l-\Theta_K}\textbf{1}_{\{\tau_{X_{\Theta_K}}^{Y}>l, Y_{l}=\rho^+\}}\right].
\end{split}
\end{equation}
By applying Fatou's Lemma and \eqref{eq_conv_b_i} on the last display of \eqref{eq_finite_positive_3}, we see that 
\begin{align*}
&\liminf_{n\rightarrow\infty}   \mathbb{E}_{(\rho^-,\rho^+)}\left[\sum_{l=0}^{n/2}\Phi(N_{(\rho^-,\rho^+)}^{Y}(l))\textbf{1}_{\{\tau_{\rho^+}^X=\infty,n-l\ge\Theta_K\}}b_{n-l-\Theta_K}\textbf{1}_{\{\tau_{X_{\Theta_K}}^{Y}>l, Y_{l}=\rho^+\}}\right]\\
\ge &\frac{1}{\mathbb E_{\rho}[\Theta_1|\tau_{\rho_*}=\infty]}\mathbb{E}_{(\rho^-,\rho^+)}\left[\sum_{l=0}^{\infty}\Phi(N_{(\rho^-,\rho^+)}^{Y}(l))\textbf{1}_{\{\tau_{\rho^+}^X=\infty\}}\textbf{1}_{\{\tau_{X_{\Theta_K}}^{Y}>l, Y_{l}=\rho^+\}}\right].
\end{align*}
It yields with \eqref{eq_finite_positive_2} that 
\begin{equation}\label{eq_finite_positive_4}
\begin{split}
     &\liminf_{n\rightarrow\infty}   \mathbb{E}_{\rho}\left[\textbf{1}_{\{\tau_{\rho_*}>n\}}\right]
        \ge\frac{1}{\mathbb E_{\rho}[\Theta_1\textbf{1}_{\{\tau_{\rho_*}=\infty\}}]}\mathbb{E}_{(\rho^-,\rho^+)}\left[\sum_{l=0}^{\infty}\Phi(N_{(\rho^-,\rho^+)}^{Y}(l))\textbf{1}_{\{\tau_{\rho^+}^X=\infty\}}\textbf{1}_{\{\tau_{X_{\Theta_K}}^{Y}>l, Y_{l}=\rho^+\}}\right]
\end{split}
\end{equation}
for all $K\ge1$. By \eqref{eq_C_expression} and the fact that $\mathbb{E}_{\rho}[\Theta_{1}|\tau_{\rho_*}=\infty]<\infty$, the right-hand side of  \eqref{eq_finite_positive_4} goes to  ${\mathscr C}/{\mathbb E_{\rho}[\Theta_1\textbf{1}_{\{\tau_{\rho_*}=\infty\}}] }$ as $K\rightarrow\infty$ by monotone convergence, which implies that 
\begin{align*}
    \mathbb E_{\rho}[\beta]\ge \frac{\mathscr C}{\mathbb E_{\rho}[\Theta_1\textbf{1}_{\{\tau_{\rho_*}=\infty\}}] }, 
\end{align*}
so $\mathscr C$ is finite. 

Next, suppose $\mathscr{C}<\infty$. By Theorem \ref{prop_asy_env}, specifically \eqref{eq_prop_asy_env_1} with $F\equiv1$ (in the proof of Theorem \ref{prop_asy_env}, we do not use the condition that $\mathbb E_{\rho}[\Theta_1\textbf{1}_{\{\tau_{\rho_*}=\infty\}}]<\infty$), we have
\begin{align*}
     \frac{\mathscr C}{\mathbb E_{\rho}[\Theta_1\textbf{1}_{\{\tau_{\rho_*}=\infty\}}]}=\mathbb E_{\rho}[\beta].
 \end{align*}
 Hence we have the finiteness of the $E_{\rho}[\Theta_1|\tau_{\rho_*}=\infty]$, which is equivalent to positive speed.
\end{proof}
\section{Speed of ERRW}\label{section_speed}	
\subsection{Speed Formula}\label{Sub_speedformula}
	
 Recall the definition of $\Phi$ and $F$ in \eqref{eq_phi_intro}  and \eqref{eq_def_F} respectively.
	
\begin{proof}[Proof of Theorem \ref{speed formula introduction}]
By dominated converge, we have 
\begin{align*}
    v=\lim_{n\rightarrow\infty} \mathbb E_{\rho}\left[\frac{|X_n|}{n}\Bigg|\mathcal{S}\right].
\end{align*}
We observe that 
\begin{align*}
     \mathbb E_{\rho}\left[|X_n|\big|\mathcal{S}\right]=\sum_{k=0}^{n-1}\mathbb E_{\rho}\left[ |X_{k+1}|-|X_k| \big|\mathcal{S}\right]=\sum_{k=0}^{n-1}\mathbb E_{\rho}\left[ \sum_{i=1}^{\nu(X_k)}\eta(X_k,X_ki)-\eta({X_k},(X_k)_{*})\Bigg| \mathcal{S}\right].
\end{align*}
We then apply Theorem \ref{prop_asy_env} and let the functional $\mathbf F$ be $\sum_{i=1}^{\nu(\rho)}\eta(\rho,\rho i)-\eta({\rho},\rho_*)$. With Lemma \ref{Lm_exp_bais}, we have
\begin{equation}\label{eq_sf_1}
    \begin{split}
          &\lim_{k\rightarrow \infty} \mathbb E_{\rho}\left[ \sum_{i=1}^{\nu(X_k)}\eta(X_k,X_ki)-\eta({X_k},(X_k)_*)\Bigg| \mathcal{S}\right]=\mathscr C^{-1}\times\\
   & \mathbb E_{(\rho^-,\rho^+)}\left[\left(\sum_{i=1}^{\nu(\rho^+)}\eta(\rho^+,\rho^+i)-\eta(\rho^+,\rho^-)\right)\frac{\beta(\rho^-)(1-\beta(\rho^+))}{\eta(\rho^+,\rho^-)}\sum_{k\geq0}\Phi(k)(1-\beta(\rho^+))^k(1-\beta(\rho^-))^k\right].
    \end{split}
\end{equation}

Recall the definition of $(A_i,1\le i\le \nu)$ in \eqref{eq_Aisratioomega}. From \eqref{eq-conductance}, we see that 
\begin{equation}\label{eq_sf_2}
    \begin{split}
          &\mathbb E_{(\rho^-,\rho^+)}\left[\left(\sum_{i=1}^{\nu(\rho^+)}\eta(\rho^+,\rho^+i)-\eta(\rho^+,\rho^-)\right)\frac{\beta(\rho^-)(1-\beta(\rho^+))}{\eta(\rho^+,\rho^-)}\sum_{k\geq0}\Phi(k)(1-\beta(\rho^+))^k(1-\beta(\rho^-))^k\right]\\
      =&\mathbb E_{(\rho^-,\rho^+)}\left[\frac{\beta_0(\sum_{i=1}^{\nu}A_i-1)}{1+\sum_{i=1}^{\nu}A_i\beta(\rho^+i)}\times \sum_{k\geq0}\Phi(k)(\frac{1-\beta(\rho^-)}{1+\sum_{i=1}^{\nu}A_i\beta(\rho^+i)})^k\right]\\=&\mathbb{E}\left[\frac{\beta_0(\sum_{i=1}^{\nu}A_i-1)}{1+\sum_{i=1}^{\nu}A_i\beta_i}\times F\left(\frac{1-\beta_0}{1+\sum_{i=1}^{\nu}A_i\beta_i}\right)\right]
    \end{split}
\end{equation}
where $\beta_0,\beta_1,\cdots$ are i.i.d copies of $\beta$. In the same way, 
\begin{equation}\label{eq_sf_3}
    \begin{split}
         \mathscr{C}=&\mathbb{E}_{(\rho^-,\rho^+)}\left[\frac{\beta(\rho^-)(1-\beta(\rho^+))}{\eta(\rho^+,\rho^-)}\sum_{k\geq0}\Phi(k)(1-\beta(\rho^+))^k(1-\beta(\rho^-))^k\right]\\
      =&\mathbb{E}\left[\frac{\beta_0(\sum_{i=1}^{\nu}A_i+1)}{1+\sum_{i=1}^{\nu}A_i\beta_i}\times F\left(\frac{1-\beta_0}{1+\sum_{i=1}^{\nu}A_i\beta_i}\right)\right].
    \end{split}
\end{equation}
Then we conclude from \eqref{eq_sf_1}-\eqref{eq_sf_3} that 
\begin{align*}
     v=&\lim_{n\rightarrow\infty}\frac{1}{n}\sum_{k=0}^{n-1}\mathbb E_{\rho}\left[ \sum_{i=1}^{\nu(X_k)}\eta(X_k,X_ki)-\eta({X_k},(X_k)_*)\Bigg| \mathcal{S}\right]\\
     =&\lim_{n\rightarrow \infty} \mathbb E_{\rho}\left[ \sum_{i=1}^{\nu(X_k)}\eta(X_k,X_ki)-\eta({X_k},(X_k)_*)\Bigg| \mathcal{S}\right]\\
     =&\mathbb{E}\left[\frac{\beta_0(\sum_{i=1}^{\nu}A_i+1)}{1+\sum_{i=1}^{\nu}A_i\beta_i}\times F\left(\frac{1-\beta_0}{1+\sum_{i=1}^{\nu}A_i\beta_i}\right)\right]^{-1}\mathbb{E}\left[\frac{\beta_0(\sum_{i=1}^{\nu}A_i-1)}{1+\sum_{i=1}^{\nu}A_i\beta_i}\times F\left(\frac{1-\beta_0}{1+\sum_{i=1}^{\nu}A_i\beta_i}\right)\right].
\end{align*}
\end{proof}

We show briefly how to get Corollary \ref{cor_formula}. Let 
\begin{align*}
    G(x)={_2F_1}(1,\alpha_c;\alpha_p;x) :=\sum_{k\geq0}\frac{\Gamma(\alpha_p)\Gamma(\alpha_c+k)}{\Gamma(\alpha_c)\Gamma(\alpha_p+k)}x^k,
\end{align*}
and
\begin{equation*}
    F(x)=\left(G(x)+\frac {xG'(x)}{\alpha_c}\right).
\end{equation*}
 By taking $\alpha_c=\alpha_p=\alpha$, $G(x)=1/(1-x)$. We can readily get the first formula in Corollary \ref{cor_formula}. When $\alpha_p=1,\alpha_c=1/2$, we can express $G$ as
\begin{equation*}
		G(x)=\frac{1}{\sqrt{1-x}},
	\end{equation*}
 by considering the expansion of binomial series
	\begin{equation*}
		\frac{1}{\sqrt{1-x}}=\sum_{k\geq0}\frac{(2k-1)!!}{(2k)!!}x^{k}.
	\end{equation*}
 Then we obtain the second formula (\ref{eq_speed_errw}) for ERRW without direction.

\subsection{Conditions for positive speed }
 
 In remaining part of the section, we set $\nu(\rho^+)$ as $\nu$, $\beta(\rho^+)$ as $\beta$, $\beta(\rho^+ i)$ as $\beta_i$, $\beta(\rho^-)$ as $\beta_0$, $\eta(\rho^+,\rho^+ i)$ as $\eta_i$ and $\eta(\rho^+,\rho^-)$ as $\eta_0$ when we work on a double tree. For a Galton-Watson tree, we also use $\nu$ for $\nu(\rho)$ and $\eta_0,\eta_i,\beta_i$ for $\eta(\rho,\rho_*),\eta(\rho,\rho i),\beta(\rho i),1\leq i\leq\nu$ respectively. Note that $A_i=\eta_i/\eta_0,1\le i\le \nu$ have a generic distribution $A$, which is independent of $\nu$. Let $\rm GW(\cdot)$ denote the Galton-Watson measure, and $\mathbb E_{GW}$ the corresponding expectation. We assume in this section that $\mathbb E_{GW}[\nu]<\infty$. We first show the moment estimation of $1/\beta$ is related to $A$. It is in fact a generalization of  \cite[Lemma 2.2]{aidekon2008transient}.
 
Recall that  $\beta(x)$ can be seen as the effective conductance of the subtree rooted at $x$, and
	\begin{equation}\label{eq_recur}
		\frac{1}{\beta}=1+\frac{1}{\sum_{i=1}^{\nu} A_i\beta_i}.
	\end{equation}

Let $\mu$ be the number of vertices in the first generation that are the roots of an infinite subtrees. We mention that $\mu$ follows the offspring distribution of another Galton-Watson tree denoted by $\mathbb T_{\mathcal{S}}$, which is made up of vertices $\{x\in\mathbb{T}:\mathbb T_x \text{ is infinite}\}$ with the generating function $\mathbb E_{GW}[s^{\mu}]=(f(q+(1-q)s)-q)/(1-q)$ \cite[Proposition 5.28]{lyons2017probability}. 
 
 Based on the knowledge above, we state a frequently used lemma. 
	\begin{lemma}\label{lemma_1overbeta}
		For a transient $(\alpha_p,\alpha_c)$-ERRW on a Galton-Watson tree and a non-negative real number $k$, we have when $p_0+p_1>0$,
		\begin{equation*}
			\mathbb E\left[\frac{\textbf{1}_{\mathcal{S}}}{\beta^{k}}\right]<\infty
		\end{equation*}
		if and only if $\mathbb E\left[A^{-k}\right]f'(q)<1$.
	\end{lemma}
\begin{remark}
  The Lemma can be extended to the random walk in random environment on a Galton-Watson tree such that $A$ is independent of $\nu$.
\end{remark}	
 \begin{proof}
		We start from the `if' part and we assume $\mathbb E\left[A^{-k}\right]f'(q)<1$. Let $\beta^{(n)}(x):={{\rm P}}^\eta_{x}(\tau_n^e<\tau_{x_*})$, where $\tau_n^e=\inf\{k\geq0:|X_k|=n\}$ is the entrance time of level $n$ and $\beta^{(n)}$ is the conductance of a tree with every vertex above level $n$ removed. All discussions about conductance still work for $\beta^{(n)}$, and in particular,
		\begin{equation*}
			\frac{1}{\beta^{(n)}(\rho)}=1+\frac{1}{\sum_{i=1}^\nu A(\rho i)\beta^{(n)}{(\rho i)}}.
		\end{equation*} 
For $N>1$, it holds for any $x>0$ that
 \begin{equation*}
     (1+x)^k\leq1+C(N)x^k\textbf{1}_{\{x> N-1\}}+N^k\textbf{1}_{\{x\leq N-1\}}\leq1+C(N)x^k+N^k
 \end{equation*} 
where $C(N)>1$ goes to $1$ as $N$ goes to infinity. Hence
		\begin{align*}
			\frac{\textbf{1}_{\mathcal{S}}}{(\beta^{(n)}(\rho))^{k}}\leq 1+C(N)\frac{\textbf{1}_{\mathcal{S}}}{(\sum_{i=1}^\nu A(\rho i)\beta^{(n)}{(\rho i)})^{k}}+N^{k}.
		\end{align*}
Let $I$ be the index of the children of $\rho$ such that $A(\rho I)=\max\{A(\rho i): 1\le i\le \nu(\rho),\beta(\rho i)>0 \}$. We have that
		\begin{equation*}
			\frac{\textbf{1}_{\mathcal{S}}}{(\sum_{i=1}^\nu A(\rho i)\beta^{(n)}{(\rho i)})^k}\leq\frac{\textbf{1}_{\{ \exists I\}}}{(A(\rho I)\beta^{(n)}(\rho I))^k}\textbf{1}_{\{A(\rho i)\beta^{(n)}(\rho i)<\epsilon,\forall i\not=I\}}+\frac{1}{\epsilon^k}.
		\end{equation*} Observe that
		\begin{equation*}
			\begin{split}
				&\lim_{\epsilon\to0}\mathbb E\left[\frac{\textbf{1}_{\{\exists I\}}}{(A(\rho I)\beta^{(n)}(\rho I))^k}\textbf{1}_{\{A(\rho i)\beta^{(n)}(\rho i)<\epsilon,\forall i\not=I\}}\right]\\
				\le &\mathbb E\left[\frac{\textbf{1}_{\mathcal{S}}}{(\beta^{(n-1)}(\rho ))^k}\right]\lim_{\epsilon\rightarrow 0} \mathbb{E} \left[\nu(\rho) A(\rho 1) ^{-k}\textbf{1}_{\{\nu(\rho) >0 \} }\textbf{1}_{\{A(\rho i)\beta(\rho i)<\epsilon,\forall i\not=1\}}\right] \\
				=&\mathbb E\left[\frac{\textbf{1}_{\mathcal{S}}}{(\beta^{(n-1)}(\rho))^k}\right]\mathbb E\left[A^{-k}(\nu q^{\nu-1}\textbf{1}_{\{\nu>1\}}+\textbf{1}_{\{\nu=1\}})\right]
			\end{split}
		\end{equation*}
by dominated convergence. Hence we have
		\begin{equation*}
			\mathbb E\left[\frac{\textbf{1}_{\mathcal{S}}}{(\beta^{(n)})^k}\right]\leq1+N^{k}+C(N)\lambda_\epsilon\mathbb E\left[\frac{\textbf{1}_{\mathcal{S}}}{(\beta^{(n-1)})^k}\right]+C(N)\epsilon^{-k},
		\end{equation*}
		where $\lambda_\epsilon:={\mathbb E}\left[A^{-k}\right]f'(q)+\delta<1/C(N)$ as long as $\epsilon$ is small enough and $N$ large enough. Finally,
		\begin{equation*}
			\mathbb E\left[\frac{\textbf{1}_{\mathcal{S}}}{(\beta^{(n)})^k}\right]\leq (1+C(N)\epsilon^{-k}+N^k)\frac1{(1-C(N)\lambda_\epsilon)}+(C(N)\lambda_\epsilon)^n(1-q).
		\end{equation*}
	We let $n$ go to infinity and apply Fatou's lemma to complete the proof.
		
  We continue to prove the `only if' part. Assume that ${\mathbb E}\left[A^{-k}\right]f^\prime(q)\ge 1$. Recall that $\mu$ is the number of children of $\rho$ where an infinite subtree is rooted. By (\ref{eq_recur}), we have that conditioned on survival,
		\begin{align*}
			\frac{1}{\beta}\geq 1+\textbf{1}_{\{\mu=1\}}\frac{1}{A_1\beta_1}.
		\end{align*}
		Therefore,
		\begin{align*}
			\frac{1}{\beta^k}> \textbf{1}_{\{\mu=1\}}\frac{1}{A_1^k\beta^k_1}.
		\end{align*}
		By taking expectations on both sides, we have
		\begin{align*}
			\mathbb E\left[\beta^{-k}|\mathcal{S}\right]> f^\prime(q){\mathbb E}\left[A_1^{-k}\right]\mathbb E\left[\beta_1^{-k}|\mathcal{S}\right],
		\end{align*}
  where $f'(q)$ comes from the fact that $\mathbb P(\mu=1|\mathcal S)=f'(q)$. Since $\beta$ and $\beta_1$ have the same distribution, and ${\mathbb E}[A^{-k}]f^\prime(q)\ge 1$ by assumption, we have $\mathbb E\left[\beta^{-k}|\mathcal{S}\right]=\infty$.
		
  \end{proof}
		Now assume ${\mathbb E}[A^{-k}]f^\prime(q)=1$. We can also deduce by induction that
		\begin{align*}
			\frac{\textbf{1}_{\mathcal{S}}}{\beta^k}\ge \textbf{1}_{\mathcal{S}}(1+\textbf{1}_{\{\mu(0)=1\}}\frac{1}{A(0)}+\textbf{1}_{\{\mu(0)=1,\mu(1)=1\}}\frac{1}{A(0)A(1)}\dots)^k,
		\end{align*}
		where $\mu(n)$, $n\geq 0$ are i.i.d with the law of $\mu$ and $A(n)$ are i.i.d with the law of $A$. 
		Set $\textbf{1}_{\{\mu(n)=1\}}A(n)$ as $\Xi_n$ for $n\ge 0$, which are i.i.d, with 
		\begin{align*}
			\mathbb E\left[(\Xi_n)^k|\mathcal S\right]=f^\prime(q) {\mathbb E}\left[\frac{1}{A^k}\right]=1.
		\end{align*}
		Then $\psi:=1+\Xi_0+\Xi_0\Xi_1\dots$ satisfies that 
		\begin{align*}
			\lim_{x\rightarrow\infty } \frac{{\mathbb P}(\psi>x|\mathcal S )}{x^{-k}}= C,
		\end{align*}
		for a positive constant $C$ by the Kesten–Grincevicius–Goldie theorem. (See, for example, \cite{kevei2016note} for a statement of the theorem.) Hence
		\begin{align}\label{eq_1/beta_ge_x}
		\liminf_{x\rightarrow\infty } \frac{{\mathbb P}(\beta^{-1}>x|\mathcal S )}{x^{-k}}\ge C.
		\end{align}
  It indicates the `heavy tail' property of $1/\beta$. 

 For an ERRW, there is analogous result when $p_0+p_1=0$. Recall that $d:=\min\{n\ge 1,p_n>0 \}$.
	\begin{lemma}\label{lemma_dgreater1}
		For a transient $(\alpha_p,\alpha_c)$-ERRW on a Galton-Watson tree such that $d\ge 2$ and a real number $k\ge 0$, 
		\begin{align*}
			\mathbb{E}\left[\frac{1}{\beta^k}\right]<\infty
		\end{align*}
		if and only if $k<d \alpha_c$.
	\end{lemma}
	\begin{proof}
		Mimic the `if' part of Lemma \ref{lemma_1overbeta}, and we have for $k<d\alpha_c$
		\begin{align*}
			&\frac{1}{(\beta^{(n)}(\rho))^{k}}\leq 1+C(N)\frac{1}{(\sum_{i=1}^\nu A(\rho i)\beta^{(n)}{(\rho i)})^{k}}+N^{k}\\
			\leq& 1+C(N)\frac{1}{(\sum_{i=1}^d A(\rho i)\beta^{(n)}{(\rho i)})^{k}}+N^{k},
		\end{align*}
	where $C(N)$ goes to $1$ as $N$ goes to infinity. Then we can prove by induction. In fact,
  $${\mathbb E}\left[\max_{1\le i\le d} A(\rho i)^{-k}\right]\le d^k {\mathbb E}\left[\left(\sum_{i=1}^d A(\rho i)\right)^{-k}\right]<\infty$$
  by Lemma \ref{prop_RWDEprop}. By taking $1\le I\le d$ that maximizes $\{A_i\}_{1\le i\le d}$, it holds that
		\begin{align*}
			&1+C(N)\frac{1}{(\sum_{i=1}^d A(\rho i)\beta^{(n)}{(\rho i)})^{k}}+N^{k}\\
			\le &1+C(N)\frac{1}{(A(\rho I)\beta^{(n)}(\rho I))^k}\textbf 1_{\{ A(\rho i)\beta^{(n)}(\rho i)<\epsilon,\,\forall i\neq I\}}+C(N)\frac{1}{\epsilon^k}+N^k
		\end{align*}
		When $\epsilon$ is small enough, we have 
		\begin{align*}
			\mathbb E\left[\frac{1}{(\beta^{(n)}(\rho))^{k}}\right]\leq  1+C(N)\delta(\epsilon)\mathbb E\left[\frac{1}{(\beta^{(n)}(\rho I))^k}\right]+C(N)\frac{1}{\epsilon^k}+N^k 
		\end{align*}
for $\delta(\epsilon):= \mathbb E\left[A(\rho I)^{-k} \textbf 1_{\{ A(\rho i)\beta^{(n-1)}(\rho i)<\epsilon,\,\forall i\neq I\}} \right]$ goes to $0$ as $\epsilon$ goes to $0$ by dominated convergence. Fix $\epsilon$ and $N$ such that $C(N)\delta(\epsilon)<1$, by the induction argument, we have that $(1/\beta^{(n)}(\rho))^k$ is integrable and so is $(1/\beta(\rho))^k$ by Fatou's lemma.
		
		On the other hand, if $k\ge d \alpha_c $,
		\begin{align*}
			\frac{1}{(\beta(\rho))^k}=\left(1+\frac{1}{\sum_{i=1}^{\nu} A(\rho i)\beta{(\rho i)}}\right)^k\geq\textbf{1}_{\{ \nu=d\}}\frac{1}{(\sum_{i=1}^{d} A(\rho i)\beta{(\rho i)})^k}\ge  \textbf{1}_{\{ \nu=d\}} \frac{1}{(\sum_{i=1}^{d} A(\rho i))^k}.
		\end{align*}
		Thus, we have $ \mathbb{E}\left[\beta^{-k}\right]=\infty$ from ${\mathbb E}\left[(\sum_{i=1}^{d} A(\rho i))^{-k}\right]=\infty$.
	\end{proof}
	
Now we are ready to investigate when $\mathscr C$ is finite. We need a technical condition that 
	\begin{align}\label{eq_moment_condition}
		\mathbb E_{GW}\left[\nu^{K}\right]<\infty,
	\end{align} 
	where $K:=d\alpha_c+3+\alpha_p$ . The condition is not necessary for the finiteness of $\mathscr C$. As is shown in Figure \ref{fig-threezone}, we deal with three different situations separately:
 \begin{itemize}
			\item $p_0=0,p_1>0$;
			\item  $p_0=0,p_1=0$;
			\item  $p_0>0$.
 \end{itemize}

First assume that $p_0=0$ and $0<p_1\le 1$. For simplicity of notation, let $s:=\alpha_c-\alpha_p+2$ and  
	\begin{align}\label{eq_def_r}
		r:=\sup\{t:{\mathbb E}\left[A^{-t}\right]p_1<1\}.
	\end{align}
The function ${\mathbb E}\left[A^{-t}\right]$ is decreasing for $t\in(-\alpha_p,(\alpha_c-\alpha_p)/2)$ and increasing for $t\in ((\alpha_c-\alpha_p)/2,\alpha_c)$. It implies that $(\alpha_c-\alpha_p)\vee 0 \le r<\alpha_c$, a fact we often need in the proof below. Now we claim the proposition in this case.
	\begin{proposition}\label{01>0}
		For a transient $(\alpha_p,\alpha_c)$-ERRW, when $p_0=0$, $p_1>0$, and (\ref{eq_moment_condition}) holds, 
		\begin{align}\label{eq_exp_finite}
			\mathbb{E}\left[\frac{\beta_0(1-\beta)}{\eta_0}\sum_{k\geq0}\dfrac{\Gamma(\alpha_p)\Gamma(\alpha_c+k+1)}{\Gamma(\alpha_c+1)\Gamma(\alpha_p+k)}(1-\beta)^k(1-\beta_0)^k\right]<\infty
		\end{align}
		if and only if  $2r-\alpha_c+\alpha_p-1>0$.
	\end{proposition}
	\begin{remark}
		In the special case $p_1=1$, where the tree degenerates to $\mathbb{Z}_+$, we need $\alpha_c>\alpha_p$ to make the walk transient. Then $r=\alpha_c-\alpha_p$ and the criterion becomes $\alpha_c-\alpha_p-1>0$, which coincides with the condition
		\begin{align*}
			{\mathbb E}\left[\frac{1}{A}\right]^{-1}=\frac{\alpha_c-1}{\alpha_p}>1
		\end{align*}
given in \cite[Theorem 1.16]{solomon1975random}.
	\end{remark}
	\begin{proof}
		For two quantities $A$ and $B$, we write $A \lesssim B$ (resp. $A \gtrsim B$) when there exists a positive constant $c$, such that $A \leq c B $ (resp. $A \geq cB$). Also, write $A\approx B$ if $A\lesssim B$ and $A \gtrsim B$. The condition $2r-\alpha_c+\alpha_p-1>0$ is equivalent to $2r-s+1>0$. 
		
		Since as $k\rightarrow\infty$,
		\begin{align*}
			\dfrac{\Gamma(\alpha_p)\Gamma(\alpha_c+k+1)}{\Gamma(\alpha_c+1)\Gamma(\alpha_p+k)} \sim k^{\alpha_c-\alpha_p+1},
		\end{align*}
		we have for $s>0$, 
		\begin{align*}
			\sum_{k\geq0}\dfrac{\Gamma(\alpha_p)\Gamma(\alpha_c+k+1)}{\Gamma(\alpha_c+1)\Gamma(\alpha_p+k)}(1-\beta)^k(1-\beta_0)^k \approx \frac{1}{(1-(1-\beta)(1-\beta_0))^s},
		\end{align*}
		for $s=0$,
		\begin{align*}
			\sum_{k\geq0}\dfrac{\Gamma(\alpha+p)\Gamma(\alpha_c+k+1)}{\Gamma(\alpha_c+1)\Gamma(\alpha_p+k)}(1-\beta)^k(1-\beta_0)^k \approx \ln (1-(1-\beta)(1-\beta_0)),
		\end{align*}
		and for $s<0$,
		\begin{align*}
			\sum_{k\geq0}\dfrac{\Gamma(\alpha+p)\Gamma(\alpha_c+k+1)}{\Gamma(\alpha_c+1)\Gamma(\alpha_p+k)}(1-\beta)^k(1-\beta_0)^k \approx 1.
		\end{align*}
		We first deal with the last case. When $s=\alpha_c-\alpha_p+2<0$, we see $\alpha_p>2$, and 
  \begin{align*}
     &\mathbb{E}\left[\frac{\beta_0(1-\beta)}{\eta_0}\sum_{k\geq0}\dfrac{\Gamma(\alpha_p)\Gamma(\alpha_c+k+1)}{\Gamma(\alpha_c+1)\Gamma(\alpha_p+k)}(1-\beta)^k(1-\beta_0)^k\right]\\
     \approx &\mathbb E\left[\frac{\beta_0(1-\beta)}{\eta_0}\right]
     \leq\mathbb E \left[\frac{1}{\eta_0}\right]\lesssim \mathbb E_{GW}\left[\nu\right] <\infty.
  \end{align*}
		In the next step, we assume $s\geq 0$. Let $s^\prime:=s+\epsilon1_{\{s=0\}}$.  First notice that
		\begin{align*}
 &\mathbb E\left[\frac{\beta_0(1-\beta)}{\eta_0(1-(1-\beta)(1-\beta_0))^{s^\prime}}\right]=\mathbb E\left[\dfrac{\beta_0(\eta_0+\sum_{i=1}^\nu \eta_i\beta_i)^{s^\prime-1}}{(\eta_0\beta_0+\sum_{i=1}^\nu \eta_i\beta_i)^{s^\prime}}\right]\\ 
   \lesssim& \mathbb E\left[\dfrac{\beta_0\eta_0^{s^\prime-1}}{(\eta_0\beta_0+\sum_{i=1}^\nu \eta_i\beta_i)^{s^\prime}}\right]+\mathbb E\left[\dfrac{\beta_0(\eta_0\beta_0+\sum_{i=1}^\nu \eta_i\beta_i)^{s^\prime-1}}{(\eta_0\beta_0+\sum_{i=1}^\nu \eta_i\beta_i)^{s^\prime}}\right]\\
			=& \mathbb E\left[\dfrac{\beta_0\eta_0^{s^\prime-1}}{(\eta_0\beta_0+\sum_{i=1}^\nu \eta_i\beta_i)^{s^\prime}}\right]+\mathbb E\left[\dfrac{\beta_0}{\eta_0\beta_0+\sum_{i=1}^\nu \eta_i\beta_i}\right]=:I_1+I_2,
		\end{align*}
  where we use $1-\beta=\eta_0/(\eta_0+\sum_{i=1}^\nu\eta_i\beta_i)$ for the first equality; $(\eta_0+\sum_{i=1}^\nu\eta_i\beta_i)^{s^\prime-1}\lesssim\eta_0^{s^\prime-1}+(\sum_{i=0}^\nu\eta_i\beta_i)^{s^\prime-1}$ for the second inequality.
  
Under the condition $2r-s+1>0$, we intend to prove that (taking $\epsilon>0$ small enough when $s=0$) $I_1+I_2<\infty$, which implies (\ref{eq_exp_finite}). By discussing whether $\eta_0$ is small, we have
		\begin{align*}
			I_1&\lesssim  \mathbb E\left[\nu^{s^\prime}\dfrac{\beta_0\eta_0^{s^\prime-1}}{(\eta_0\beta_0+\beta_1)^{s^\prime}}1_{\{\eta_0\le 1/2\}}\bigg|\eta_1=\max_{1\leq i\leq\nu}\{\eta_i\}\right]+\mathbb E\left[\dfrac{\beta_0}{(\beta_0+\eta_1\beta_1)^{s^\prime}}1_{\{\eta_0\ge 1/2\}}\right]\\
  &= \mathbb E\left[\nu^{s^\prime}\dfrac{\beta_0\eta_0^{s^\prime-1}}{(\eta_0\beta_0+\beta_1)^{s^\prime}}1_{\{\eta_0\le 1/2\}}\right]+\mathbb E\left[\dfrac{\beta_0}{(\beta_0+\eta_1\beta_1)^{s^\prime}}1_{\{\eta_0\ge 1/2\}}\right] =:I_3+I_4,
		\end{align*}
  where we use the observation that $\{\eta_0\le 1/2\}$ implies that $\{\max_{1\leq i\leq\nu}\eta_i>1/2\nu\}$. 
  
  For $I_3$, it holds that
   \begin{align*}
			I_3\lesssim  \mathbb E\left[\nu^{s^\prime}\dfrac{\beta_0\eta_0^{s^\prime-1}}{(\eta_0\beta_0+\beta_1)^{s^\prime}}\right]\lesssim \mathbb E \left[\nu^{s^\prime}  \frac{1}{\beta_0^{s^\prime-1}\eta_0} \textbf{1}_{\{\beta_1\leq\eta_0\beta_0\}}
			+\nu^{s^\prime} \frac{\beta_0\eta_0^{s^\prime-1}}{\beta_1^{s^\prime}}\textbf{1}_{\{\beta_1>\eta_0\beta_0\}}\right].
		\end{align*}
  Since $\mathbb E[(\beta_1)^{-r+\delta}]<\infty$ for any positive $\delta$ by Lemma \ref{lemma_1overbeta}, $\mathbb P(\beta_1  \le x)\lesssim x^{r-\delta}$ by Markov inequality. The density of $\eta_0\in (0,1)$ is given by
\begin{align}\label{eq_eta_0}
\frac{\d \mathbb P(\eta_0\le w|\nu) }{\d w}=\frac{\Gamma(\alpha_p+\nu\alpha_c)}{\Gamma(\nu\alpha_c)\Gamma(\alpha_p)}w^{\alpha_p-1}(1-w)^{\nu\alpha_c-1}  \lesssim \frac{\Gamma(\alpha_p+\nu\alpha_c)}{\Gamma(\nu\alpha_c)\Gamma(\alpha_p)}w^{\alpha_p-1}\approx\nu^{\alpha_p}w^{\alpha_p-1}
\end{align}
by (\ref{eq_density_Dirichlet}) and Stirling's approximation. Hence, we have
		\begin{align*}
			&\mathbb E\left[\nu^{s^\prime}  \frac{1}{\beta_0^{s^\prime-1}\eta_0} \textbf{1}_{\{\beta_1\leq\eta_0\beta_0\}}\right]=\mathbb E_{GW}\left[\nu^{s^\prime}  \int_{x\le zw} \frac{1}{z^{s^\prime-1}w} \d \mathbb P(\beta_0\le z)\d  \mathbb P(\eta_0 \le w|\nu) \d \mathbb P(\beta_1 \le x)\right] \\
			\lesssim &\mathbb E_{GW}\left[\nu^{s^\prime+\alpha_p}  \int z^{r-s^\prime+1-\delta}w^{r+\alpha_p-\delta-2} \d \mathbb P(\beta_0\le z)\d w\right]\lesssim \mathbb E_{GW}\left[\nu^{s^\prime+\alpha_p}\right]<\infty,
		\end{align*}
	where the last approximation follows from $2r-s'+1>0$ (and thus $r+\alpha_p-1>\alpha_c-r>0$) by taking $\epsilon$ and $\delta$ small enough. Similarly, together with integration by parts, we have
		\begin{align*}
			&\mathbb E\left[\nu^{s^\prime} \frac{\beta_0\eta_0^{s^\prime-1}}{\beta_1^{s^\prime}}\textbf{1}_{\{\beta_1>\eta_0\beta_0\}}\right]=\mathbb E_{GW}\left[\nu^{s^\prime}\int_{x>zw}\frac{zw^{s^\prime-1}}{x^{s^\prime}}  \d \mathbb P(\beta_0\le z)\d \mathbb P(\eta_0 \le w|\nu) \d \mathbb P(\beta_1\le x)\right]\\
			\lesssim 
			&\mathbb E\left[\nu^{s^\prime}\beta_0\eta_0^{s'-1}\right]+\mathbb E_{GW}\left[\nu^{s^\prime}\int_{x>zw} \frac{zw^{s^\prime-1}}{x^{s^\prime+1}} \mathbb P(\beta_1 \le x) \d x  \d \mathbb P(\beta_0\le z)\d \mathbb P(\eta_0 \le w|\nu)\right]\\
			\lesssim&\mathbb E_{GW}\left[ \nu^{s^\prime}\int zw^{s^\prime-1} (z^{r-s^\prime-\delta}w^{r-s^\prime-\delta}\vee 1) \d \mathbb P(\beta_0\le z)\d \mathbb P(\eta_0 \le w|\nu)\right]\lesssim \mathbb E_{GW}\left[\nu^{s^\prime+\alpha_p}\right]<\infty,
		\end{align*}
where we use the fact that $\mathbb E\left[\nu^{s^\prime}\beta_0\eta_0^{s'-1}\right]$ is finite since $\alpha_p+s'-1\ge \alpha_c+1>0$ and $\mathbb E[\nu]\le\mathbb E[\nu^{\alpha_c+3+\alpha_p}]<\infty$.

  For $I_4$, the same deduction yields that 
		\begin{align*}
			&\mathbb E\left[\frac{\beta_0}{(\beta_0+\eta_1\beta_1)^{s^\prime}}\right]\lesssim\mathbb E \left[ \frac{\beta_0}{\beta_1^{s^\prime}\eta_1^{s^\prime}} \textbf{1}_{\{\beta_0\leq\eta_1\beta_1\}}
			+\frac{1}{\beta_0^{s^\prime-1}}\textbf{1}_{\{\beta_0>\eta_1\beta_1\}}\right]\\
   \lesssim & \mathbb  E_{GW}\left[\int w^{r-s^\prime+1-\delta} z^{r-s^\prime+1-\delta}\d\mathbb  P(\beta_1\le z) \d \mathbb P(\eta_1\le w|\nu)\right]\lesssim  \mathbb  E_{GW}\left[\nu^{\alpha_c}\right]<\infty,
		\end{align*}
	since $\mathbb E[(\beta_1)^{-r+\delta}]<\infty$,
  \begin{align}\label{eq_eta_1}
       \frac{\d\mathbb P(\eta_1\leq w|\nu)}{\d w}\lesssim \frac{\Gamma(\alpha_p+\nu\alpha_c)}{\Gamma(\alpha_c)\Gamma((\nu-1)\alpha_c+\alpha_p)}w^{\alpha_c-1}\approx\nu^{\alpha_c}w^{\alpha_c-1}
  \end{align}
  and $2r-s'+1>0$ (which implies $\alpha_c+r-s'+1>0$ due to $\alpha_c>r$).
		
		We then deal with $I_2$. We claim that $I_2$ is integrable under the condition $2r-\alpha_c+\alpha_p-1>0$. In fact, if $r>0$, we see $2r-1+1>0$ and $r+\alpha_p-1>\alpha_c-r>0$, By plugging $s^\prime=1$ into $I_1$, we obtain $I_2<\infty$ immediately. Note that when $s^\prime=1$, we have $\mathbb E_{GW}\left[\nu^{1+\alpha_p}\right]\le \mathbb E_{GW}\left[\nu^{3+\alpha_c+\alpha_p}\right]<\infty$. If $r=0$, we have $\nu\equiv1$ and $\alpha_c\le\alpha_p$, which falls into the region of recurrence.
  
  Now we continue to prove that $2r-s+1>0$ is necessary by contradiction. It is true even if the condition (\ref{eq_moment_condition}) is dropped. Assume that $2r-s+1\le 0$ (which implies $s\ge 1$). It yields that 
  \begin{align*}
     &\mathbb{E}\left[\frac{\beta_0(1-\beta)}{\eta_0}\sum_{k\geq0}\dfrac{\Gamma(\alpha_p)\Gamma(\alpha_c+k+1)}{\Gamma(\alpha_c+1)\Gamma(\alpha_p+k)}(1-\beta)^k(1-\beta_0)^k\right]\\
     \approx& \mathbb E\left[\dfrac{\beta_0(\eta_0+\sum_{i=1}^\nu \eta_i\beta_i)^{s-1}}{(\eta_0\beta_0+\sum_{i=1}^\nu \eta_i\beta_i)^{s}}\right]
     \gtrsim  \mathbb E\left[\dfrac{\beta_0\eta_0^{s-1}}{(\eta_0\beta_0+\sum_{i=1}^\nu \eta_i\beta_i)^{s}}\right]=I_1
  \end{align*}
  
If $s=1$, $r=0$, which implies recurrence. If $s>1$, since event $\{\nu=1\}$ has positive probability, we have
		\begin{align*}
			I_1&\gtrsim\int \frac{zw^{s-1}}{(zw+x)^{s}} \d \mathbb P(\beta_0\le z)\d \mathbb P(\eta_0 \le w) \d \mathbb P(\beta_1 \le x)\\
			&\gtrsim \int \dfrac{w^{s-1}}{(1+w)^{s}}\d \mathbb P(\eta_0 \le w)\int_{x<z} z^{-s+1}\d \mathbb P(\beta_0\le z) \d \mathbb P(\beta_1 \le x)\\
			&\gtrsim \mathbb E\left[\beta_1^{-s+1}\wedge\beta_0^{-s+1}\right]\gtrsim\int_{z\ge 1}z^{s-2}\mathbb P({\beta_1}^{-1}\wedge{\beta_0}^{-1}\geq z)\d z= \int_{z\le 1} \mathbb P(\beta_0\le z )^2z^{-s}\d z, 
		\end{align*}
		since $\beta_0$ and $\beta_1$ have the same law. On the other hand, as is shown in  (\ref{eq_1/beta_ge_x}), there is a positive constant $C$ such that
		\begin{align*}
			\liminf_{x\rightarrow \infty}\frac{\mathbb P({\beta}^{-1}>x)}{x^{-r}}\ge C.
		\end{align*}
		Therefore, $\int_{z\le 1} \mathbb P^2(\beta_0\le z) z^{-s}\d z=\infty$ since $2r-s+1\le 0$, which implies that \eqref{eq_exp_finite} does not hold. 
	
	\end{proof}
	We continue to deal with the case when $p_0=0$, $p_1=0$. Recall $d:=\min\{n\ge 1: p_n>0\}$ and $\mathbb E[\beta^{-k}]<\infty$ if and only if $k<d\alpha_c$ by Lemma \ref{lemma_dgreater1}.
	\begin{proposition}\label{01=0}
		For a transient $(\alpha_p,\alpha_c)$-ERRW, when $p_0=0$, $p_1=0$, and (\ref{eq_moment_condition}) holds,
		\begin{align*}
			\mathbb{E}\left[\frac{\beta_0(1-\beta)}{\eta_0}\sum_{k\geq0}\dfrac{\Gamma(\alpha_p)\Gamma(\alpha_c+k+1)}{\Gamma(\alpha_c+1)\Gamma(\alpha_p+k)}(1-\beta)^k(1-\beta_0)^k\right]<\infty,
		\end{align*}
		if and only if $(2d-1)\alpha_c+\alpha_p-1> 0$.
	\end{proposition}
	\begin{proof}
		Follow the notation in Proposition \ref{01>0}. Recall that $1-\beta=\eta_0/(\eta_0+\sum_{i=1}^\nu\eta_i\beta_i)$ and 
  \begin{align*}
     \frac{\d \mathbb P(1-\eta_1\leq x|\nu)}{\d x}=\frac{\Gamma(\alpha_p+\nu\alpha_c)}{\Gamma(\alpha_c)\Gamma((\nu-1)\alpha_c+\alpha_p)}x^{(\nu-1)\alpha_c+\alpha_p-1}(1-x)^{\alpha_c-1}. 
  \end{align*}  
  To prove the necessity, suppose $(2d-1)\alpha_c+\alpha_p-1\leq0$, we see 
		\begin{align*}
			& \mathbb{E}\left[\frac{\beta_0(1-\beta)}{\eta_0}\sum_{k\geq0}\dfrac{\Gamma(\alpha_p)\Gamma(\alpha_c+k+1)}{\Gamma(\alpha_c+1)\Gamma(\alpha_p+k)}(1-\beta)^k(1-\beta_0)^k\right]\\
   \gtrsim &\mathbb  E\left[\frac{1-\beta}{\eta_0}\right]\gtrsim \mathbb E\left[\frac{\textbf{1}_{\{\nu=d\}}}{\eta_0+\sum_{i=1}^d \eta_i\beta_i}\right]\\
			\gtrsim & \int \frac{1}{x+y}\d \mathbb P(\eta_0+\sum_2^d\eta_i\le x|\nu=d )\mathbb \d \mathbb P(\beta_1\le y)\\
   \gtrsim & \int_{x\le y\leq 1/2} \frac{1}{y}\d \mathbb P(1-\eta_1\le x |\nu=d)\mathbb \d \mathbb P(\beta_1\le y)\\
			\gtrsim & \int_{y\leq1/2} y^{(d-1)\alpha_c+\alpha_p-1}\d \mathbb P(\beta_1\le y)=\infty
		\end{align*}
  since $(d-1)\alpha_c+\alpha_p-1\leq -d\alpha_c$. 
  
Following the proof in Proposition \ref{01>0}, we only need to prove for sufficiency that when $(2d-1)\alpha_c+\alpha_p-1> 0$ and $\mathbb E\left[\nu^{K}\right]<\infty$ (recall that $K=d\alpha_c+3+\alpha_p$),
  \begin{align*}
      I_1=\mathbb E\left[\dfrac{\beta_0\eta_0^{s^\prime-1}}{(\eta_0\beta_0+\sum_{i=1}^\nu \eta_i\beta_i)^{s^\prime}}\right]<\infty
  \end{align*}
where $s^\prime=s+\textbf{1}_{\{s=0\}}\epsilon$. (Note that $I_2$ is finite automatically by taking $s^\prime=1$ in $I_1$.) By separating the case $\eta_0> 1/2$ and $\eta_0\le 1/2$, we have by symmetry of $\eta_i,1\le i\le \nu$,
  \begin{align*}
      I_1&\lesssim\mathbb E\left[\dfrac{\beta_0\textbf{1}_{\{\eta_0>1/2\}}}{(\beta_0+\sum_{i=1}^\nu \eta_i\beta_i)^{s^\prime}}\right]+\mathbb E\left[\nu\dfrac{\beta_0\eta_0^{s^\prime-1}\textbf{1}_{\{\eta_1=\max\{\eta_i,1\le i\le \nu\},\eta_0<1/2 \}}}{(\beta_0\eta_0+\sum_{i=1}^\nu \eta_i\beta_i)^{s^\prime}}\right]\\
      &\lesssim\mathbb E\left[\dfrac{\beta_0\textbf{1}_{\{\eta_0>1/2\}}}{(\beta_0+\sum_{i=1}^\nu \eta_i\beta_i)^{s^\prime}}\right]+\mathbb E\left[\nu^{s^\prime+1}\dfrac{\beta_0\eta_0^{s^\prime-1}\textbf{1}_{\{\eta_1=\max\{\eta_i,1\le i\le \nu\},\eta_0<1/2 \}}}{(\beta_0\eta_0+\beta_1+\sum_{i=2}^\nu \eta_i\beta_i)^{s^\prime}}\right]\\
      &\lesssim\mathbb E\left[\dfrac{\beta_0\textbf{1}_{\{\eta_0>1/2\}}}{(\beta_0+\sum_{i=1}^\nu \eta_i\beta_i)^{s^\prime}}\right]+\mathbb E\left[\nu^{s^\prime+1}\dfrac{\beta_0\eta_0^{s^\prime-1}\textbf{1}_{\{\eta_1>1/(2\nu)\}}}{(\beta_0\eta_0+\beta_1+\sum_{i=2}^\nu \eta_i\beta_i)^{s^\prime}}\right]\\
      &\lesssim \mathbb E\left[\dfrac{\beta_0\textbf{1}_{\{\eta_0>1/2\}}}{(\beta_0+\sum_{i=1}^d \eta_i\beta_i)^{s^\prime}}\right]+\mathbb E\left[\nu^{s^\prime+1}\dfrac{\beta_0\eta_0^{s^\prime-1}\textbf{1}_{\{\eta_1>1/(2\nu)\}}}{(\beta_0\eta_0+\beta_1+\sum_{i=2}^d \eta_i\beta_i)^{s^\prime}}\right]=:I_3+I_4.
  \end{align*}
For $I_4$, it holds by the definition of Dirichlet distribution (\ref{eq_density_Dirichlet}) that
\begin{align*}
     &\mathbb E\left[\nu^{s^\prime+1}\dfrac{\beta_0\eta_0^{s^\prime-1}\textbf{1}_{\{\eta_1>1/(2\nu)\}}}{(\beta_0\eta_0+\beta_1+\sum_{i=2}^d \eta_i\beta_i)^{s^\prime}}\right]\\
    \lesssim &\mathbb E\Bigg[\nu^{s^\prime+1}\dfrac{\Gamma(\alpha_p+\nu\alpha_c)}{\Gamma((\nu-d+1)\alpha_c)}\int_{1-w-\sum_{i=2}^d x_i\ge \frac{1}{2\nu}}\dfrac{\beta_0w^{s^\prime-1}}{(\beta_0w+\beta_1+\sum_{i=2}^dx_i\beta_i)^{s^\prime}}\\
    &[(\frac{1}{2\nu})^{(\nu-d+1)\alpha_c-1 }\vee 1]w^{\alpha_p-1}\d w \prod_{i=2}^d (x_i^{\alpha_c-1}\d x_i) \Bigg]\\
     \approx&\mathbb E\left[\nu^{s^\prime+1+(d-1)\alpha_c+\alpha_p}\dfrac{\beta_0\bar\eta_0^{s^\prime-1}}{(\beta_0\bar \eta_0+\beta_1+\sum_{i=2}^d \bar\eta_i\beta_i)^{s^\prime}}\right],
\end{align*}
where the density of $\bar{\eta}_0$ is proportional to $w^{\alpha_p-1},0<w<1$ and that of $\bar{\eta}_i, 2\le i\le d$ are proportional to $w^{\alpha_c-1},0<w<1$. The random variables $\bar\eta_i,2\le i\le d$, and $\bar\eta_0$ are independent of each other and of other random variables. In the first inequality, we use the fact that when $(\nu-d+1)\alpha_c-1<0$, on the event that $\{\eta_1>1/(2\nu)\}$,
\begin{align*}
    \left(1-w-\sum_{i=2}^d x_i\right)^{(\nu-d+1)\alpha_c-1}\le \left(\frac{1}{2\nu}\right)^{(\nu-d+1)\alpha_c-1}.
\end{align*}
Note that by independence
\begin{align*}
    \mathbb P\left(\beta_1+\sum_{i=2}^d \bar \eta_i\beta_i<x\right)\le \mathbb P(\max\{\beta_1,\bar \eta_2\beta_2,\dots\bar \eta_d\beta_d \}<x)= \mathbb P(\beta_1\le x)\prod_{i=2}^d \mathbb P({\bar\eta_i}\beta_i\le x).
\end{align*}
Thus we have $\mathbb P(\beta_1+\sum_{i=2}^d \bar \eta_i\beta_i<x)\lesssim  x^{(2d-1)\alpha_c-\delta}$ from $\mathbb P(\beta\le y)\lesssim y^{d\alpha_c-\delta}$ and $\mathbb P({\bar\eta_i}\beta_i<x)=\mathbb E[\mathbb P(\bar\eta_i<x/\beta_i|\beta_i)]\lesssim   x^{\alpha_c},2\leq i\leq d$. Therefore, it holds by the same deduction as in Proposition \ref{01>0} that
\begin{align*}
    &\mathbb E\left[\nu^{s^\prime+1+(d-1)\alpha_c+\alpha_p}\dfrac{\beta_0\bar \eta_0^{s^\prime-1}}{(\beta_0\bar \eta_0+\beta_1+\sum_{i=2}^d \bar\eta_i\beta_i)^{s^\prime}}\right]\\
    \lesssim & \mathbb E_{GW}\left[\nu^{s^\prime+1+(d-1)\alpha_c+\alpha_p}\int\dfrac{zw^{s^\prime-1}}{(zw+x)^{s^\prime}}\d \mathbb P(\beta_0\le z)\d \mathbb P(\bar \eta_0\le w)\d \mathbb P(\beta_1+\sum_{i=2}^d \bar\eta_i\beta_i\le x|\nu) \right]\\
    \lesssim &\mathbb E_{GW}\left[\nu^{d\alpha_c+3+\epsilon}\int z^{(2d-1)\alpha_c-\delta+1-s-\epsilon} w^{(2d-1)\alpha_c-\delta+\alpha_p-2}\d \mathbb P(\beta_0\le z)\d w\right]
\end{align*}
which is finite when $(2d-1)\alpha_c+\alpha_p-1>0$ (and thus $(3d-2)\alpha_c+\alpha_p-1>0$) and $\mathbb E[\nu^{d\alpha_c+3+\epsilon}]\le\mathbb E[\nu^{d\alpha_c+3+\alpha_p}]<\infty$ when $\epsilon$ small enough.

We then deal with $I_3$ in the same way. By the substitution mentioned above,  
\begin{align*}
    \mathbb E\left[ \dfrac{\beta_0\textbf{1}_{\{\eta_0>1/2\}}}{(\beta_0+\sum_{i=1}^d \eta_i\beta_i)^{s^\prime}}\right]\lesssim \mathbb E_{GW}\left[\nu^{d\alpha_c}\int\dfrac{z}{(z+x)^{s^\prime}}\d \mathbb P(\beta_0\le z)\d \mathbb P(\sum_{i=1}^d{\bar\eta_i}\beta_i\le x)\right].
\end{align*}
Recall that $\mathbb P(\beta\le y)\lesssim y^{d\alpha_c-\delta}$ by Lemma \ref{lemma_dgreater1}  and $\mathbb P(\bar\eta_i\le x_i)\lesssim  x_i^{\alpha_c}$ by definition. Then we have $\mathbb P(\sum_{i=1}^d{\bar\eta_i}\beta_i<x)\leq \mathbb P(\max_{1\leq i\leq d}\{{\bar\eta_i}\beta_i\}<x)\lesssim   x^{d\alpha_c}$ from $\mathbb P({\bar\eta_i}\beta_i<x)=\mathbb E[\mathbb P(\bar\eta_i<x/\beta_i|\beta_i)]\lesssim  x^{\alpha_c},1\le i\le d$ and independence of $(\bar\eta_i\beta_i)_{1\le i\le d}$. Following the proof of Proposition \ref{01>0} , for some $\delta>0$ small enough, 
\begin{align*}
    \mathbb E_{GW}\left[\nu^{d\alpha_c}\int\dfrac{z}{(z+x)^{s^\prime}}\d \mathbb P(\beta_0\le z)\d \mathbb P(\sum_{i=1}^d{\bar\eta_i}\beta_i\le x)\right]\lesssim  \mathbb E_{GW}\left[\nu^{d\alpha_c}\int z^{d\alpha_c+1-s^\prime}\d\mathbb P(\beta_0\le z) \right],
\end{align*}
which is finite when $\mathbb E[\nu^{d\alpha_c}]<\infty$ and $(2d-1)\alpha_c+\alpha_p-1>0$. 
\end{proof}
	
	Finally, we assume that $p_0>0$. Define
	\begin{align*}
		r:=\sup\{k:{\mathbb E}[A^{-k}]f^\prime(q)<1\}.
	\end{align*}
	We also have $0<r<\alpha_c$ from $f^\prime(q)<1$.
	
	\begin{proposition}\label{0>0}
		For a transient $(\alpha_p,\alpha_c)$-ERRW, when $p_0>0$ and (\ref{eq_moment_condition}) holds, 
		\begin{align}\label{eq_exp_p0>0}
			\mathbb E\left[\frac{\beta_0(1-\beta)}{\eta_0}\sum_{k\geq0}\dfrac{\Gamma(\alpha_p)\Gamma(\alpha_c+k+1)}{\Gamma(\alpha_c+1)\Gamma(\alpha_p+k)}(1-\beta)^k(1-\beta_0)^k \Bigg| \mathcal{S}_{\mathbb{T}^-}\right]<\infty,
		\end{align}
		if and only if $r-\alpha_c+\alpha_p-1>0$.
	\end{proposition}
	\begin{proof}
	Let us first deal with the `if' part. The inequality (\ref{eq_exp_p0>0}) holds if
		\begin{align*}
			\mathbb  E\left[\frac{\beta_0}{\eta_0}\sum_{k\geq0}\dfrac{\Gamma(\alpha_p)\Gamma(\alpha_c+k+1)}{\Gamma(\alpha_c+1)\Gamma(\alpha_p+k)}(1-\beta_0)^k \right]<\infty.
		\end{align*}
		If $s=\alpha_c-\alpha_p+2<0$, then the summation is bounded and we only need to ensure that $\mathbb E[1/\eta_0]<\infty$. It is naturally satisfied since $r-\alpha_c+\alpha_p-1>0$ (which implies $\alpha_p>1$ by $r<\alpha_c$), and $\mathbb E_{GW}[\nu]<\infty$. When $s\geq 0$, let us consider $s'=s+\epsilon\textbf{1}_{\{s=0\}}>0$ for $\epsilon>0$ small enough. We see
		\begin{align*}
			\mathbb  E\left[\frac{\beta_0}{\eta_0}\sum_{k\geq0}\dfrac{\Gamma(\alpha+p)\Gamma(\alpha_c+k+1)}{\Gamma(\alpha_c+1)\Gamma(\alpha_p+k)}(1-\beta_0)^k \right]\lesssim\mathbb  E[\frac{\beta_0}{\eta_0(1-(1-\beta_0))^{s^\prime}}]={\mathbb E}\left[\eta_0^{-1}\right]\mathbb E\left[\beta_0^{1-s^\prime}\right]
		\end{align*}
is finite when $r-\alpha_c+\alpha_p-1>0$ (which implies $\alpha_p>1$ and $1-s'>-r$) and $\mathbb E_{GW}[\nu]<\infty$.

For the `only if ' part, when $r-\alpha_c+\alpha_p-1\le 0$, we have $s\geq r+1>0$.  Since $p_0>0$, $\mathbb P(\beta=0)>0$. On the event $\{\nu=0\}$, the left-hand side of (\ref{eq_exp_p0>0}) becomes
\begin{align*}
   &\mathbb  E\left[\frac{\beta_0}{\eta_0}\sum_{k\geq0}\dfrac{\Gamma(\alpha+p)\Gamma(\alpha_c+k+1)}{\Gamma(\alpha_c+1)\Gamma(\alpha_p+k)}(1-\beta_0)^k \textbf{1}_{\{\nu=0\}}\right]
   \approx\mathbb  E\left[\frac{\beta_0}{(1-(1-\beta_0))^{s}}\right]\approx \mathbb E\left[\beta_0^{1-s}\right] 
\end{align*}
which is infinite by Lemma \ref{lemma_1overbeta}. 
	\end{proof}
 Now Theorem \ref{positive speed Thm} is obtained from Proposition \ref{01>0}, \ref{01=0} and \ref{0>0}.

	\bibliographystyle{plain}
	\bibliography{main}
\end{document}